\documentclass[12pt,reqno]{amsart}
\usepackage{amsbsy,amsfonts,amsmath,amssymb,amscd,amsthm}
\usepackage{mathrsfs,float}
\usepackage[mathcal]{euscript}
\usepackage{subfigure,enumitem,graphicx,epstopdf,color}
\usepackage[a4paper,text={6.1in,8in},centering,marginparwidth=23mm]{geometry}
\usepackage{pxfonts,epstopdf}
\usepackage[colorlinks=true,linkcolor=blue,citecolor=green,dvipdfmx]{hyperref}
\usepackage{srcltx}
\usepackage[margin=20pt,font=small,labelfont=bf,textfont=it]{caption}
\usepackage{comment}
\usepackage{mathtools}
\usepackage[all]{xy}
\usepackage{tikz} 
\usepackage{footmisc}

\usepackage{graphicx}

\usepackage{ascmac}
\usepackage{here}

\theoremstyle{plain}
\newtheorem{mainthm}{Theorem} 
\newtheorem{thm}{Theorem}[section]

\newtheorem{lem}[thm]{Lemma}
\newtheorem{prop}[thm]{Proposition}

\newtheorem{problem}[thm]{Problem}
\theoremstyle{definition}
\newtheorem{dfn}[thm]{Definition}

\newtheorem{example}[thm]{Example}

\theoremstyle{remark}
\newtheorem{remark}[thm]{Remark}

\numberwithin{equation}{section}

\newcommand{\N}{\mathbb{N}}

\newcommand{\Z}{\mathbb{Z}}
\newcommand{\C}{\mathbb{C}}
\newcommand{\vep}{\varepsilon}
\newcommand{\Ga}{\Gamma}
\newcommand{\wmO}{\widetilde{\mathcal{O}}}
\newcommand{\mA}{\mathcal{A}}

\newcommand{\acts}{\curvearrowright}
\newcommand{\T}{\mathbb{T}}

\newcommand{\lv}{\left\lvert}

\newcommand{\rv}{\right\rvert}

\newcommand{\ga}{\gamma}
\newcommand{\mH}{\mathcal{H}}
\newcommand{\mO}{\mathcal{O}}

\small\normalsize
\let\oldmarginpar\marginpar
\renewcommand\marginpar[1]{\-\oldmarginpar[\raggedleft\tiny #1]%
{\raggedright\tiny #1}}

  \title[Topological entropy for    Exel--Laca algebras]{Topological entropy for   countable Markov shifts and Exel--Laca algebras}

\author[Michimoto]{Yuta Michimoto}
\address[Yuta Michimoto]{Nippon Hy\= oronsha, Tokyo, 170-8474, JAPAN}
\email{roadbook21150@gmail.com}

\author[Nakano]{Yushi Nakano}
\address[Yushi Nakano]{Department of Mathematics, Tokai University,  Kanagawa, 259-1292, JAPAN}
\email{yushi.nakano@tsc.u-tokai.ac.jp}

\author[Toyokawa]{Hisayoshi Toyokawa}
\address[Hisayoshi Toyokawa]{Faculty of Engineering, Kitami Institute of Technology, Hokkaido, 090-8507, JAPAN}
\email{h\_toyokawa@mail.kitami-it.ac.jp}

\author[Yoshida]{Keisuke Yoshida}
\address[Keisuke Yoshida]{Department of Mathematics, Faculty of Science, Hokkaido University, Hokkaido, 060-0810, JAPAN}
\email{kskyhuni@math.sci.hokudai.ac.jp}

    \date{}


\begin{document}
   \begin{abstract}
We show that the (Gurevich) topological entropy for the countable Markov shift associated with an infinite transition matrix $A$  coincides with the non-commutative topological entropy for the Exel--Laca algebra associated with $A$, under certain conditions on $A$.
An important example satisfying the conditions is the renewal shift, which is not locally finite.
We also pose interesting questions for future research on non-commutative topological entropy for non-locally finite transition matrices.
\end{abstract}
  \maketitle
  
  \tableofcontents

\section{Introduction}
The topological entropy was first introduced in 1965 by Adler, Konheim and McAndrew \cite{AKM1965} for continuous maps on a compact metric space as a criterion to measure how  complicated  the global behavior of the dynamical system is, and enjoyed a great success in ergodic theory (such as mathematically rigorous foundation of statistical mechanics \cite{R1999,W2000}, thermodynamical formalism in fractal geometry \cite{PU2010}, or a key tool in non-uniformly hyperbolic systems theory \cite{BDV2004}).
Hence, it was a natural hope to extend the concept of topological entropy to operator algebras.
Among several attempts, Voiculescu \cite{V1995} in 1995 introduced his useful definition of non-commutative entropy for  completely positive maps on nuclear C$^*$-algebras, and it was later extended to exact C$^*$-algebras by Brown \cite{B1999}.
We refer to monographs \cite{NS2006,RS2013} for backgrounds of non-commutative entropy.

In   \cite{C1996}, Choda showed that the topological entropy for full shifts with a finite alphabet coincides with the non-commutative topological entropy for the associated Cuntz algebra. 
This was extended by Boca and Goldstein in \cite{BG2000} to finite Markov shifts and their associated Cuntz--Krieger algebras. 
On the other hand, Matsumoto \cite{M2005} gave examples of subshifts on a finite state whose topological entropy is \emph{strictly smaller} than the non-commutative topological entropy for the associated Matsumoto algebra. 
In this paper, we consider another important generalization of finite Markov shifts, i.e.~countable Markov shifts, and establish that its (Gurevich) topological entropy coincides with the non-commutative topological entropy for  the associated Exel--Laca algebra, under certain  conditions on the transition matrix and KMS states:
\begin{mainthm} \label{thm:main}
Let $\Sigma _A$ and $\mathcal O_A$ be the countable Markov shift and Exel--Laca algebra associated with a $0$-$1$ matrix $A$.
Let $\sigma _A: \Sigma _A\to \Sigma _A$ be the left shift operation and $\gamma _A:\mathcal O_A\to \mathcal O_A$ be the canonical completely positive map.
Assume that {\tt (SH)}, {\tt (FS)}, {\tt (AF)},  {\tt (SI)}, {\tt (SD)} and {\tt (O)} hold for $A$ and there exists a KMS state of the canonical gauge action on $\mO _A$.
Then, 
\begin{equation*}
ht (\gamma _A) = h_{\mathrm{G}}(\sigma _A).
\end{equation*}
\end{mainthm}

All terminologies in Theorem \ref{thm:main} will be explained in Sections \ref{s:def}, \ref{s:can} and \ref{s:main}.
In particular, we define the canonical completely positive map $\gamma _A$ in Section \ref{s:can}, which is far from trivial in our setting where $A$ is not necessarily locally finite. 
We will give our main theorem in a more complete form in Section \ref{s:main}, as well as several interesting questions for future research in Section \ref{s:problems}.
\begin{remark}\label{rem:1227c}
After we completed the draft version of the proof of Theorem \ref{thm:main}, we learned that in a series of papers \cite{JP04,JP06,JP09} Jeong and Park studied the relation between the block and loop entropies of a directed graph and the non-commutative entropy of the associated graph C$^*$-algebra.
Using the $*$-isomorphism between the Exel--Laca algebra and its associated graph C$^*$-algebra (cf.~\cite[Proposition 4.8]{Renault00}), one may get that 
\[
h_{\mathrm{G}}(\sigma _A) \le ht (\gamma _A) \leq h_{b}(A^{tr})
\]
 under their assumptions on $A$, including local finiteness of $A$, where $h_{b}(A^{tr})$ is the (Salam) block entropy of the graph induced by the transpose matrix $A^{tr}$ of $A$.

Although the works by Jeong and Park are important and overlap a part of our result, the present paper includes several novelties essentially because they always assumed that $A$ is locally finite, while we do not. 
Throughout this paper, one may realize that the absence of (local) finiteness requires one to investigate a quite newer technology, even in the definition of the canonical completely positive map.
 An application after Theorem \ref{thm:main} is the renewal shift, an important non-locally finite class appearing in the study of intermittent dynamics (cf.~\cite{Sarig2001}).

We also remark that as a by-product of the proof of Theorem \ref{thm:main}, we give new information on the difference between $h_G(\sigma _A)$ and $h_b(A^{tr})$.\footnote{ 
In Examples 3.8 and 3.9 of \cite{JP09}, it was shown that $h_l(A) < h_b(A^{tr})$ holds for some matrix $A$ representing a locally finite graph allowing multi-edges from a vertex to another vertex, where $h_{l}(A)$ is the (Gurevich) loop entropy of the graph induced by $A$. 
It holds that $h_l(A) = h_G(\sigma _A)$ when $A$ is a transition matrix, but their matrix $A$ is not so.}
In fact, in Remark \ref{rm:1227b}, we show that $h_G(\sigma _A) \le h_b(A^{tr}) \le h_G(\sigma _A) + \mathcal I$ with a quantity $\mathcal I$ associated with $A$ (refer to Definition \ref{dfn:SI}), as well as giving a locally finite transition matrix $A$ with $\mathcal I>0$.
 The condition $\mathcal I=0$ is called {\tt (SI)} in Theorem \ref{thm:main}, so {\tt (SI)} is also a sufficient condition for $ht(\gamma _A) = h_G(\sigma _A) $ in the locally finite case.
\end{remark}

\section{Preliminaries}
\label{s:def}
In this section, we prepare some terminologies that appeared in Theorem \ref{thm:main}.
\subsection{Countable Markov shifts}
Let $D$ be a countable set. 
Denote by $D^{\mathbb N}$ the one-sided infinite product of $D$. 
Let $\sigma$ be the left shift operation of $D^{\mathbb N}$ (i.e.~$(\sigma (x))_j = x_{j+1}$ for each $j\in \mathbb N$ and $x= \{x_j\}_{j\in \mathbb N} \in D^{\mathbb N}$). 
When a subset  $\Sigma $ of $D^{\mathbb N}$ is $\sigma$-invariant and  closed, we call it a \emph{subshift}, and $D$  the \emph{alphabet} of $\Sigma $.
We equip $\Sigma $ with the topology generated by the base of cylinder sets 
\[
[x]_n\coloneqq \{ \{y_j\}_{j\in\mathbb N}\in \Sigma   : x_j =y_j \; \text{for all $j=1, 2,\ldots ,n$} \}
\]
 over $x=\{x_j\}_{j\in\mathbb N}\in \Sigma $ and $n\in \mathbb N$.
We define the set of \emph{admissible} finite words 
\[
\Sigma ^*\coloneqq \bigcup _{n\geq 0} \left\{ \{x_j\}_{j=1}^n\in D^n : \text{$\exists y= \{y_j\}_{j\in \mathbb N} \in \Sigma$ such that $x_j =y_j$ for all $1\leq j\leq n$}\right\},
\]
where $\{x_j\}_{j=1}^0$ means the empty word, denoted by $e$.
We write $\vert x \vert \coloneqq  n$ and $t(\alpha)\coloneqq  x_n$ for each $x= \{x_j\}_{j=1}^n\in \Sigma ^*$, and write $\vert e\vert \coloneqq 0$ and $t(e)\coloneqq 1$.
The cylinder set over an admissible finite word $x= \{x_j\}_{j=1}^n$ is also denoted by $[x]$.
Namely,
\[
[x] \coloneqq \{ \{y_j\}_{j\in\mathbb N}\in \Sigma   : x_j =y_j \; \text{for all $j=1, 2,\ldots ,n$} \}.
\]

When 
 $\Sigma $ is of the form
\[
\Sigma  =\{ x\in D^{\mathbb N} : \text{$A(x_j, x_{j+1}) =1 $ for all $j\in \mathbb N$}\}
\]
with  a $0$-$1$ matrix $A= \{A(i,j)\}_{(i,j)\in D^2}$ (that is,  each entry of which is $0$ or $1$), we call $\Sigma $ a \emph{Markov shift}.
When we emphasize the dependence of $\Sigma $ on $A$, it is denoted by $\Sigma _A$, and $A$ is called the \emph{transition matrix} of $\Sigma _A$.
We  denote the restriction of $\sigma$ on $\Sigma _A$ by $\sigma_A$.
Observe that $\Sigma _A^*=\bigcup _{n\geq 0} \left\{ \{x_j\}_{j=1}^n\in D^n : \text{$A(x_j, x_{j+1}) =1$ for all $1\leq j\leq n-1$}\right\}.$

We say that a $0$-$1$ matrix $A=\{A(i,j)\}_{(i,j)\in D^2}$ is \emph{finite} if $\vert D\vert <\infty$, where $\vert B\vert$ denotes the cardinality of a set $B$. 
Moreover, $A$ is said to be \emph{locally finite} if for any $i\in D$, $\vert \{j\in D : A(i,j)=1\}\vert <\infty$ and $\vert \{j\in D : A(j,i)=1\} \vert <\infty$.
If in addition the numbers are bounded uniformly with respect to $i\in D$, then $\Sigma _A$ is said to be \emph{uniformly locally finite}.
An important example of countable Markov shift induced by a matrix being not locally finite is the \emph{renewal shift} $\Sigma _A$, given as $D=\mathbb N$ and 
\begin{equation}\label{eq:renewal}
A(i,j) =
\begin{cases}
1 \quad &  (\text{$i=1$ or $j=i-1$})\\
0 \quad & (\text{otherwise})
\end{cases}.
\end{equation}
The renewal shift naturally appears in the study of intermittent dynamics, and its thermodynamic formalism, including the phase transition, was intensively studied by Sarig \cite{Sarig1999,Sarig2001}.
Furthermore, recently it was sharpened by  Bissacot et al.~\cite{BEFR2018} from the viewpoint of C$^*$-algebras. 
As another application of our main result, we also consider
\begin{equation}\label{eq:lazyrenewal}
A(i,j) =
\begin{cases}
1 \quad &  (\text{$i=1$ or $j\in \{ i-1, i\}$})\\
0 \quad & (\text{otherwise})
\end{cases}
\end{equation}
with $D=\mathbb N$, that we call the \emph{lazy renewal shift} (named after the lazy random walk).
    \begin{figure}[h]
    \centering
        \includegraphics[width=12truecm]{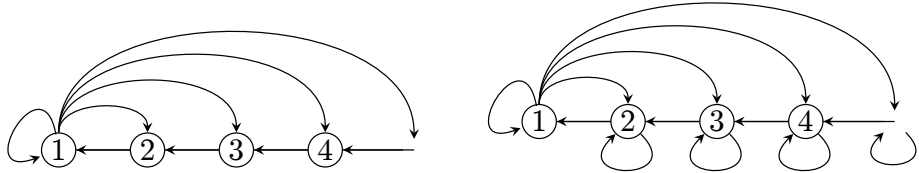}
        \vspace{0.4cm}
        \caption{{\rm Renewal shift (left); lazy renewal shift (right)}}\label{fig:lazyrenewal}
    \end{figure}

\subsection{Gurevich topological entropy}\label{ss:Ge}
It is well known that the classical definition of topological entropy $h_{\mathrm{top}}(\sigma ) $ of a continuous map $\sigma :X\to X$   on a   topological space $X$ (in the sense of the original definition by Adler, Konheim and McAndrew \cite{AKM1965}) does not behave well if $X$ is not compact.
In fact, it follows from Theorem 1.3 of \cite{HNP2008} that if  there is a point $x\in X$ whose orbit $\{ \sigma ^n(x)\} _{n=0}^\infty$  does not have an accumulation point (e.g.~$\sigma $ is the renewal shift), then $h_{\mathrm{top}}(\sigma ) =\infty$.

A useful definition of topological entropy for countable Markov shifts was given by Gurevich \cite{Gurevich1969} from the viewpoint of the variational principle, and its properties were later intensively studied by Sarig \cite{Sarig1999}:
If $\sigma _A: \Sigma_A\to \Sigma_A$ is topologically mixing,\footnote{A continuous map $\sigma : X\to X$ on a metric space is said to be \emph{topologically mixing} if for any nonempty open sets $U, V\subset X$ one can find $N\in \mathbb N$ such that $U\cap \sigma ^{-n} V\neq \emptyset$ for any $n\geq N$.} then the following holds (cf.~\cite{Sariglecturenote}).
\begin{itemize}
\item The \emph{Gurevich topological entropy} $h_{\mathrm{G}}(\sigma _A)$ of $\sigma _A$, given by
\[
h_{\mathrm{G}}(\sigma _A)\coloneqq  \lim _{n\to\infty} \frac{1}{n} \log  \sum _{\sigma _A^n(x)=x} 1_{[a]_1}(x)
\]
for some $a\in \Sigma _A$, exists and is independent of the choice of $a$.
\item The variational principle holds:
\[
h_G(\sigma _A)= \sup \left\{ h_\mu (\sigma _A) \mid \text{$\mu$: $\sigma_A$-invariant Borel probability measure}\right\},
\]
where $h_\mu (\sigma _A)$ is the metric entropy of $\mu$.
\end{itemize}

\subsection{Non-commutative topological entropy}\label{subsec23}
Let $\mathcal O$ be an exact C$^*$-algebra and $\pi : \mathcal O \to B(\mH)$ a faithful $*$-representation of $\mathcal O$ on the set $B(\mH)$ of bounded linear operators on a Hilbert space $\mH$.
Given a finite subset $\omega \subset \mathcal O$ and $\delta >0$, the completely positive $\delta$-rank $rcp(\omega ;\delta ) \equiv rcp_\pi (\omega ;\delta )$ of $\omega$ is given by
\begin{multline*}
rcp(\omega ;\delta )=\min \Big\{ N\in\mathbb N : \text{there exists $(\phi , \psi , \mathcal A) \in CPA(\mathcal O)$ such that} \\
\mathrm{rank}\, \mathcal A =N, \quad \Vert \psi \circ \phi (S) -\pi (S) \Vert <\delta \; \text{for all $S\in \omega$} \Big\},
\end{multline*}
where $CPA(\mathcal O)$ is the set of triples $(\phi , \psi , \mathcal A)$ of a finite-dimensional C$^*$-algebra $\mathcal A$ and cp (i.e.~completely positive) maps $\phi :\mathcal O\to \mathcal A$, $\psi :\mathcal A\to B(\mH)$, $ \mathrm{rank}\, \mathcal A$ is the dimension of a maximal abelian C$^*$-subalgebra of $\mathcal A$, and $\Vert \cdot \Vert$ is the norm on $B(\mH )$.
Standard references for non-commutative entropy are e.g.~\cite{RS2013, NS2006}.
By Kirchberg--Wassermann's nuclear embedding of exact C$^*$-algebra (cf.~\cite{W1994}), there exists a faithful $*$-representation $\pi : \mathcal O \to B(\mH )$ such that $rcp _\pi (\omega ;\delta )<\infty$ for all finite subset $\omega \subset \mathcal O$ and $\delta >0$.
Furthermore, it is observed by Brown \cite{B1999} that $rcp(\omega ;\delta )$ is independent of the choice of $\pi$.

Let $\gamma :\mathcal O \to \mathcal O$ be a cp map. 
The non-commutative topological entropy $ht(\gamma )$ of $\gamma$ is given by
\[
ht(\gamma ) = \sup_\omega ht(\gamma ,\omega ), \qquad ht(\gamma ,\omega ) = \sup _{ \delta >0} \limsup _{n\to \infty} \frac{1}{n} \log rcp\left( \bigcup _{j=0}^{n-1} \gamma ^j (\omega ); \delta \right),
\]
where the supremum is taken over all finite subsets $\omega$ of $\mathcal O$. 
The basic properties of the non-commutative topological entropy are as follows, refer to \cite{BG2000, NS2006}.
\begin{itemize}
\item If $\mathcal O_0$ is a subalgebra of $\mathcal O$ such that $\gamma (\mathcal O_0)\subset \mathcal O_0$, then $ht (\gamma \vert _{\mathcal O_0}) \le ht (\gamma )$.
\item If $\{\omega _j\}_{j\in\N}$ is an increasing sequence of finite subsets of $\mathcal O$ (i.e.~$\omega _j\subset \omega _{j+1}$ for any $j\in \N$) and the norm closure of the linear span of $\bigcup _{j, n\in \N} \gamma ^n (\omega _j)$ equals to $\mathcal O$, then $ht(\gamma ) =\sup _{j\in \N} ht(\gamma ,\omega _j).$
\item If $\theta$ is an automorphism of $\mathcal O$, then $ht (\theta \circ \gamma \circ \theta ^{-1}) = ht (\gamma )$.
\item Given a continuous map $\sigma : X\to X$ on a compact metric space $X$, if $\mathcal O=C(X)$ and $\gamma (f)=f\circ \sigma$ for each $f\in \mathcal O$, then $ht (\gamma ) =h_{\mathrm{top}}(\sigma )$.
\end{itemize}

\subsection{KMS state and GNS representation}
Let $\mathcal O$ be a unital C$^*$-algebra and $\widehat\Gamma =\{\widehat\Gamma _t\}_{t\in \mathbb R}$ a strongly continuous one-parameter group of *-authmorphisms of $\mathcal O$.
A state $\varphi$ on $\mathcal O$ (i.e.~a positive linear functional on $\mathcal O$ with $\varphi (1)=1$) is called a \emph{KMS$_\beta$ state} for $\widehat\Gamma$ with an inverse temperature parameter $\beta >0$ if
\[
\varphi (ST)=\varphi \left(T \, \widehat \Gamma_{i\beta }\left(S\right)\right)
\]
for any entire analytic element $S\in\mathcal O$ (i.e.~the map $t\mapsto \widehat\Gamma _t(S)$ from $\mathbb R$ to $\mathcal O$ extends to an entire analytic function) and any $T\in\mathcal O$.
Furthermore, given a group action $\Gamma :\mathbb T\acts \mathcal O$, a state $\varphi$ on $\mathcal O$ is called a KMS state for $\Gamma$ if $\{ \Gamma _{it}\}_{t\in \mathbb R}$ is a KMS$_\beta$ state for some $\beta >0$.
KMS states can be seen as a generalization of Gibbs states and played an important role in the context of quantum statistical mechanics, refer to e.g.~\cite{NS2006}. 
Furthermore, KMS states were recently used in \cite{BEFR2018} to sharpen the phase transition about conformal measures for countable Markov shifts.

Given a state $\varphi$ on $\mathcal O$, one can introduce an inner product $\langle \cdot , \cdot \rangle  $ on $\mathcal O/\mathcal N _\varphi$, where $\mathcal N _\varphi \coloneqq \{ S\in\mathcal O : \varphi (S^* S)=0\}$, by $\langle S+ \mathcal N _\varphi , T+\mathcal N _\varphi  \rangle \coloneqq \varphi (S^* T)$.
Denote by $\mathcal H_\varphi$ the Hilbert space given by the completion of $\mathcal O/\mathcal N _\varphi$.
Define a *-homomorphism $\pi _\varphi : \mathcal O\to B(\mathcal H_\varphi )$ by $\pi (S)(T+\mathcal N_\varphi )= ST+\mathcal N_\varphi $.
Then, the pair $(\pi _\varphi , \mathcal H_\varphi )$ is called the \emph{GNS representation} of $\varphi$.
The GNS representation of a KMS state is a fundamental (von Neumann algebraic) approach in operator algebras theory, refer to \cite{BR2012}.

\subsection{Exel--Laca C$^*$ algebra}\label{ss:EL}
For a 0-1 matrix $A=\{A(i,j)\}_{(i,j)\in D^2}$ with a countable symbol $D$, the \emph{Exel--Laca algebra} $\mO_A$ associated with $A$ is the universal unital C$^*$-algebra generated by partial isometries $\{S_j\}_{j\in D}$ satisfying the following condition:
\begin{equation}\label{EL1}
(S^* _iS_i)( S^* _j S_j) = (S^* _j S_j)(S^* _iS_i );
\end{equation}
\begin{equation}\label{EL2}
S^* _iS_j = \delta_{i, j}S^* _iS_i ;
\end{equation}
\begin{equation}\label{EL3}
(S^* _iS_i)S_j = A(i, j)S_j ;
\end{equation}
\begin{equation}\label{EL4}
\prod_{i \in X} S^* _iS_i \prod_{j \in Y} (1 - S^* _j S_j) = \sum_{k=1}^\infty A(X, Y, k)S_kS^*_k  
\end{equation}
for any $i, j \in D$ and any finite subset $X, Y \subset D$ such that 
\[
\left|\{ k \in D : A(X, Y, k) =1 \}\right| < \infty, 
\]
where
\[
A(X, Y, k)=
\begin{cases}
1 \quad &(\text{if $A(i,k) =1$, $A(j,k) =0$ for all $i\in X$, $j\in Y$)}\\
0 \quad &(\text{otherwise})
\end{cases}.
\]
Exel and Laca first introduced the C$^*$-algebra $\mO_A$ in \cite{EL1999}.
In the paper, they also considered the universal \textit{unital} C$^*$-algebra $\wmO_A$ generated by isometries $\{ S_j\}_{j\in D}$ and showed $\mO _A \simeq \wmO_A$ if and only if $\mO_A$ is unital.
However, we only consider $\mO _A$ because it is straightforward to see that if $\mO_A$ is non-unital and {\tt (FS)} (assumed in Theorem \ref{thm:main}; refer to Definition \ref{dfn:fs}) holds, then $A$ is row finite,\footnote{It follows from Lemma \ref{finsum} that if $R_j^c$ is a finite set for some $j\in \N$, then $1=S^*_jS_j + \sum _{k \in R_j^c} S_k S_k^*$. Since the sum in the right-hand side is a finite sum, this implies that $\mO_A$ is unital. Thus, if $\mO_A$ is non-unital and {\tt (FS)} holds, then $R_j$ is a finite set for any $j\in \N$.\label{f:3a}}
so that $A$ is locally finite under {\tt (UCF)} (refer to Definition \ref{dfn:ucf}).

The basic properties of Exel--Laca algebras are as follows, refer to \cite{EL1999, Raszeja}.
\begin{itemize}
\item If $A$ is a finite matrix, then $\mO_A$ is isomorphic to the Cuntz--Krieger algebra associated with $A$.
\item $\mO_A$ is nuclear (in particular, exact).
\item If $A$ is an infinite matrix, has no zero rows (that is, there does not exist an $i\in D$ such that $A(i,j) =0$ for any $j \in D$), and $\sigma _A:\Sigma _A\to \Sigma _A$ is mixing, then the Exel--Laca algebra $\mathcal O_A$ is simple and unique.
\item The norm closure of a *-subalgebra $\mathcal{A}_A$ of $\mO_A$, given by
\[
\mathcal{A}_A \coloneqq \mathrm{span}\left\{ S_\alpha \left( \prod_{i \in F}S^* _i S_i \right)S^* _\beta : \alpha, \beta \in \Sigma _A^*, \; F \subset D \ \text{is a finite subset} \right\},
\]
where $S_e\coloneqq 1$ and $S_\alpha \coloneqq  S_{\alpha _1}S_{\alpha _{2}} \cdots S_{\alpha _n}$ for $\alpha =\{ \alpha _j\}_{j=1}^n$, equals to $\mO_A$.
\end{itemize}
Notice also that the mixing property of $\sigma _A: \Sigma _A\to \Sigma _A$ immediately implies that $A$ has no zero rows and has no zero columns (i.e.~there does not exist an $j\in D$ such that $A(i,j) =0$ for any $i \in D$).
Furthermore, as previously mentioned, it is already shown in \cite{BG2000} that $ht(\gamma _A) = h_{\mathrm{top} }(\Sigma _A)$ if $A$ is a finite matrix, where $\gamma _A$ is the canonical ucp map (see Section \ref{s:can}) on the Cuntz--Krieger algebra of $A$.
Taking them and the properties in Sections \ref{ss:Ge} and \ref{ss:EL} into account, similar to \cite{Raszeja}, we hereafter assume the following condition.
\begin{dfn}
A $0$-$1$ matrix $A=\{A(i,j)\}_{(i,j)\in D^2}$ is said to satisfy the \textit{standing hypothesis} {\tt (SH)} if $D=\mathbb N$,  $\sigma _A:\Sigma _A\to \Sigma _A$ is mixing and $\mO_A$ is unital.
\end{dfn}

\section{Canonical ucp map on the Exel--Laca algebra}\label{s:can}
In this section, we define a ucp (i.e.~unital cp) map $\ga_A \colon \mO_A \to \mO_A$ corresponding to $\sigma _A:\Sigma _A\to \Sigma_A$ for a $0$-$1$ matrix $A=\{A(i,j)\}_{(i,j)\in \N ^2}$, as well as providing a part of the definitions of the assumptions in Theorem \ref{thm:main}.
First, we aim to define
\begin{equation}\label{eq:can1}
\ga_A (S_j S^*_j)\coloneqq \sum_{i=1}^\infty S_iS_jS^* _j S^* _i
\end{equation}
for each $j \in \N$ (which naturally appears in the study of $\sigma _A$ when respecting $\mathcal O_A$, as indicated in Section \ref{subsec31}).
Since there is a faithful representation $\pi :\mO_A \to B(\mH)$ on a Hilbert space $\mH$, one can define the infinite sum $\sum_{i=1}^\infty S_iS_jS_j ^*S_i ^*$ using the strong operator topology on  $B(\mH)$.
Although this definition looks natural, we might not have $\ga_A(S_iS_i ^*) \in \mO_A$ in general since $\mO_A$ is not closed in the strong operator topology.
To avoid treating infinite sums, the following condition plays an indispensable role.
For notational simplicity, we write
\[
C_j \coloneqq  \{ i \in \N :  A(i, j) =1\}, \quad R_j \coloneqq  \{ i \in \N : A(j, i) = 1\}.
\]
\begin{dfn}\label{dfn:ucf}
A $0$-$1$ matrix $A=\{A(i,j)\}_{(i,j)\in \mathbb N^2}$ is said to be \textit{uniformly column finite} {\tt (UCF)} if we have
\[
M_A\coloneqq \max_{j \in \N}|C_j| < \infty.
\]
\end{dfn}
An important property of a uniformly column finite 0-1 matrix $A$ is that $h_G(\sigma _A)<\infty $, refer to \cite{Sariglecturenote}.

We introduce the canonical gauge action $\Ga \colon \T \acts \mO_A$ given by
\[
\Ga_z (S_j)\coloneqq zS_j
\]
for each $z \in \T$ and $j \in \N$.  
Note that for any $z\in \T$ the sequence of isometries $\{ zS_j\}_{j \in \N}$ satisfies the equations \eqref{EL1}-\eqref{EL4}.
Hence,
$\Ga_z$ is a well-defined surjective $*$-endomorphism on $\mO_A$.
In fact, since $\mO_A$ is simple due to {\tt (SH)},  $\Ga_z$ is an automorphism.
Write 
\[
(\mO_A)^\Ga \coloneqq \{ T \in \mO_A : \Ga_z(T)=T \ \text{for any} \ z \in \T \}. 
\]
We have $\Ga_z(S_\alpha S^* _\beta) = z^{|\alpha| - |\beta|}S_\alpha S^* _\beta$ for any $\alpha, \beta \in \Sigma _A^*$ and $z \in \T$.
Thus, for any $\alpha, \beta\in \Sigma _A^*$, we see that $S_\alpha S^* _\beta \in (\mO_A)^\Ga$ if and only if $|\alpha| = |\beta|$.
Hence, since $\mathcal O_A$ is the norm closure of $\mathcal A_A$, we get
\begin{equation}\label{eq:1221b}
\left(\mO_A\right)^\Ga = \overline{\text{span}}\left\{ S_\alpha \left( \prod_{i \in F}S^* _i S_i \right)S^* _\beta : \alpha, \beta \in \Sigma _A^*, \; |\alpha| = |\beta|, \; F \subset \N \ \text{is a finite subset} \right\}.
\end{equation}
Under this observation, we introduce the commutative C$^*$-subalgebra $\mathcal D_A$ of $\mathcal O_A$ by
\[
\mathcal D_A = \overline{\text{span}}\left\{ S_\alpha \left( \prod_{i \in F}S^* _i S_i \right)S^* _\alpha : \alpha  \in \Sigma _A^*, \; F \subset \N \ \text{is a finite subset} \right\}.
\]

To define $\ga_A (S^* _j S_j)$ by \eqref{eq:can1} for each $j\in \N$, we need a condition under which $S^*_j S_j$ can be expressed as a linear combination of $1 \in \mO_A$ and a finite sum of $S_iS^* _i$'s.
Denote the complement of a subset $B\subset \mathbb N$ by $B^c$.

\begin{lem}\label{finsum}
For any $j\in \N$, the following holds.
\begin{enumerate}
\item
If $R_j$ is a finite set, then one has
\[
S^* _jS_j = \sum_{k\in R_j}S_kS^* _k.
\] 
\item
If $R_j ^c$ is a finite set, then one has
\[
S^* _jS_j = 1 - \sum_{k\in R_j ^c}S_kS^* _k.
\] 
\end{enumerate}
\end{lem}
\begin{proof}
Apply \eqref{EL4} to $X= \{j\}$ and $Y= \emptyset$.
Then we have
\[
S^* _jS_j = \sum_{k\in R_j}S_kS^* _k,
\]
that is, (1) is proven.

Apply \eqref{EL4} to $X= \emptyset$ and $Y= \{ j \}$.
Then 
\[
1 - S^* _jS_j = \sum_{k\in R^ c _j}S_kS^* _k.
\]    
This immediately implies (2).
\end{proof}

Consider the canonical faithful $*$-representation $\pi : \mO_A\to B(l^2(\Sigma _A))$ of $\mO _A$,\footnote{The canonical representation $\pi : \mO_A\to B(l^2(\Sigma _A))$ is given by $\pi (S_i) =T_i$ ($i\in \mathbb N$) with 
\[
T_i(\delta _x) = 
\begin{cases}
\delta _{ix} \quad &(A(i, x_1) =1)\\
0 & (\text{otherwise})
\end{cases}, \quad 
T_i^*(\delta _x) = 
\begin{cases}
\delta _{\sigma (x)} \quad &(x_1 =i)\\
0 & (\text{otherwise})
\end{cases}
\]
for $x=\{ x_j\}_{j=1}^m\in \Sigma _A$, see \cite{Raszeja}. \label{footnote_3}} 
and simply write $S_j$ for $\pi (S_j)$.
Then we have 
\begin{equation}\label{eq:1111d}
\sum_{i=1}^\infty S_iS^* _i= 1_{B(l^2(\Sigma_A))}
\end{equation}
in the strong operator topology on $B(l^2(\Sigma_A))$.
Respecting this, we define $\ga_A(1)\coloneqq 1$ later.
If $A$ is uniformly column finite, then $\ga_A(\sum_{j \in F}S_jS^*_j)$ is given by a finite sum for any finite subset $F \subset \N$.
Thus, for any $j \in \N$ with $\min\{|R_j|, |R_j ^c|\} < \infty$, one can define $\ga _A(S^* _jS_j)$ using Lemma \ref{finsum} and $\ga_A(1)\coloneqq 1$.
Therefore, the following condition allows us to define $\ga_A(S^*_j S_j)$ for each $j \in \N$.

\begin{dfn}\label{dfn:fs}
A natural number $j $ is said to be an \textit{$A$-finite emitter} (resp.~an \textit{$A$-infinite emitter}) if $R_j$ (resp.~$R_j ^c$) is a finite set.
We say that $A$ is \textit{finitely summable} {\tt (FS)} if 
\[
\N = \{ j \in \N : j \ \text{is an $A$-finite emitter} \} \cup \{ j \in \N : j \ \text{is an $A$-infinite emitter} \}.
\]
\end{dfn}

Notice that under {\tt (FS)}, $\mathcal O_A$ is unital if and only if there is an $A$-infinite emitter (refer to footnote \ref{f:3a}).

\begin{lem} 
Assume that {\tt (UCF)} holds for $A$. 
Then the number of $A$-infinite emitters is at most $M_A$.
\end{lem}

\begin{proof}
Suppose that there exists a subset
\[
J \subset \{ j \in \N : j \ \text{is $A$-infinite emitter}\}
\]
with $|J|=M_A +1$.
Then one can find $N \in \N$ with $A(j, i)=1$ for any $i \geq N, j \in J$.
Hence, we have
\[
C_N = \{ k \in \N : A(k, N) = 1\} > M_A .
\] 
This is a contradiction.
\end{proof}

For convenience, we write 
\begin{equation}\label{eq:1117e}
\Omega_A \coloneqq \{ S_\alpha S^* _\beta : \alpha, \beta \in \Sigma _A^* \}\subset \mO_A, \quad
\omega_A \coloneqq \{ S_\alpha S^* _\beta \in \Omega_A : |\alpha| = |\beta| \}.
\end{equation}

\begin{lem}
Assume that {\tt (FS)} holds for $A$. 
Then we have
\[
\mathcal{A}_A = \text{\upshape{span}} \, \Omega_A.
\] 
\end{lem}

\begin{proof}
First, we check that $\text{span} \, \Omega_A$ is a $*$-algebra.
It is trivial that $\text{span} \, \Omega_A$ is a $*$-closed vector subspace of $\mO_A$. 
Take $S_{\alpha}S^ *_{\beta}, \, S_{\alpha'}S^ *_{\beta'} \in \Omega_A$ with $S_{\alpha}S^ *_{\beta}S_{\alpha'}S^ *_{\beta'} \neq 0$. 
We show $S_{\alpha}S^ *_{\beta}S_{\alpha'}S^ *_{\beta'} \in \text{span} \, \Omega_A$. 

\textbf{Case I : $|\beta| > |\alpha|$.} \quad There is an admissible word $\gamma$ with $S^ *_{\beta}S_{\alpha'} = S^* _\gamma$.
Hence, we have $S_{\alpha}S^ *_{\beta}S_{\alpha'}S^ *_{\beta'} \in \Omega_A \subset \text{span} \, \Omega_A$.

\textbf{Case II: $\beta = \alpha'$.} \quad Using Lemma \ref{finsum}, we have either 
\[
S^* _{\beta} S_{\alpha'} = S^* _{t(\beta)}S _{t(\beta)} = \sum_{k \in R_{t(\beta)}}S_kS^* _k 
\]
or 
\[
S^* _{\beta} S_{\alpha'} = S^* _{t(\beta)}S _{t(\beta)} = 1 - \sum_{k \in R^c _{t(\beta)}}S_kS^* _k .
\]
Both of two equalities imply $S_{\alpha}S^ *_{\beta}S_{\alpha'}S^ *_{\beta'} \in \text{span} \, \Omega_A$.

Thus, we have $S_{\alpha}S^ *_{\beta}S_{\alpha'}S^ *_{\beta'} \in \text{span} \, \Omega_A$ and therefore $\text{span}\, \Omega_A$ is a $*$-subalgebra of $\mA_A$. 
Hence, it suffices to show that $S^* _jS_j \in \text{span}\, \Omega_A$ for any $j \in \N$ since $\mathcal{A}_A$ is spanned by the elements of the form $S_\alpha \left( \prod_{i \in F}S^* _i S_i \right)S^* _\beta$. 
By Lemma \ref{finsum}, we obtain the conclusion. 
\end{proof}

\begin{dfn}
Suppose that {\tt (UCF)} and {\tt (FS)} hold for $A$.
We define a linear map $\ga_A \colon \mathcal{A}_A \to \mathcal{A}_A$ by
\[
\ga_A(S_\alpha S^* _\beta) \coloneqq  \sum_{i=1}^\infty S_iS_\alpha S^* _\beta S^* _i, \quad \ga_A(1) \coloneqq  1
\]
for each $S_\alpha S^* _\beta \in \Omega_A\setminus \{1\}$.
\end{dfn}

Notice that $\sum_{i=1}^\infty S_iS_\alpha S^* _\beta S^* _i$ equals to a finite sum by {\tt (UCF)}.

\begin{prop}\label{conti}
Suppose that {\tt (UCF)} and {\tt (FS)} hold for $A$.
Then $\ga_A$ is norm-continuous.
\end{prop} 

\begin{proof}
We show $\| \ga_A(f)\| \leq \| f \|$ for any $f \in \mathcal{A}_A$.
We disassemble $f$ into the form 
\[
f =f_0 + c 1
\]
where $f_0 \in \text{span}\, \Omega_A$ and $c \in \C$. 
Consider the canonical representation $\pi \colon \mO_A \to B(l^2(\Sigma_A))$.
Fix an arbitrary $\vep > 0$.
Choose $\xi \in l^2(\Sigma_A)$ with $\| \xi \| =1$ and $\| \ga_A(f) \xi \| \geq \| \ga_A(f) \| - \vep$. 
Since $\bigoplus_{i \in \N} \text{Im} S_iS^* _i = l^2(\Sigma_A)$, there is a natural number $n_0$ with 
\[
\left\| \xi - \sum_{i=1}^n S_iS^* _i \xi \right\| < \vep
\]
for any $n \geq n_0$.
Choose $n_1 \in \N$ large enough that we have 
\[
\ga_A(f_0) = \sum_{i=1}^{n_1}S_i f_0S^* _i.
\] 
Set $n_2 \coloneqq  \max\{n_0, n_1\}$.
Then we get 
\begin{equation*}
\begin{split}
\left\| \left(\ga_A(f) - \sum_{i=1}^{n_2}S_i fS^* _i \right) \xi \right\| 
&= \left\| \left( \sum_{i=1}^{n_1}S_i f_0 S^* _i + c 1 - \sum_{i=1}^{n_1}S_i f_0S^* _i - c\sum_{i=1}^{n_2}S_i S^* _i \right) \xi\right\| \\
&= |c| \left\| \left( 1 - \sum_{i=1}^{n_2}S_i S^* _i \right) \xi \right\| \leq |c| \vep.  
\end{split}
\end{equation*}
Thus, we have
\[
\| \ga_A(f)\xi \| \leq \left\| \left(\sum_{i=1}^{n_2}S_i fS^* _i  \right)\xi \right\| + |c|\vep.
\]
Note that $\{ S_i fS^* _i \xi \}_{i \in \N}$ and $\{ S_i S^* _i \xi \}_{i \in \N}$ are sequences of mutually orthogonal vectors in $l^2(\Sigma_A)$.
Thus,
\begin{equation*}
\left\| \left(\sum_{i=1}^{n_2}S_i fS^* _i\right) \xi \right\|^2 = \sum_{i=1}^{n_2} \| S_i fS^* _i \xi\|^2 = \sum_{i=1}^{n_2} \|S_i fS^* _i S_i S^* _i \xi \|^2 \leq \|f \|^2 \sum_{i=1}^{n_2} \| S_i S^* _i \xi \|^2 \leq \| f\|^2.
\end{equation*}
Hence, we obtain
\[
\| \ga_A(f) \| \leq \| \ga_A(f) \xi \| + \vep \leq \|f\| + (1+|c|)\vep.
\]
Since $\vep >0$ is arbitrary, we get the conclusion.
\end{proof}

From now on we always assume that  {\tt (SH)}, {\tt (UCF)} and {\tt (FS)} hold for $A$. 
Proposition \ref{conti} allows us to extend $\ga_A$ to a bounded linear map on $\mO_A$.
We use the same symbol $\ga_A$ for the extended map.
Note that one has 
\[
\ga_A (f)= \lim_{n \to \infty}\sum_{i=1}^n S_i f S^* _i
\]
in the point-strong operator topology for any $f\in \mathcal O_A$  (under the canonical representation $\pi : \mO_A\to B(l^2(\Sigma_A))$ in the footnote \ref{footnote_3}).
Thus, $\ga_A \colon \mO_A \to \mO_A$ is a ucp map since each $f\mapsto \sum_{i=1}^nS_i f S^* _i$ is a ccp (i.e.~contractive cp) map.

\subsection{Relation between $\sigma _A$ and $\gamma _A$}\label{subsec31}

In this subsection, we relate the countable Markov shift $\sigma_A$ acting on the shift space $\Sigma_A$ and the ucp map $\gamma _A$ over $\mathcal{D}_A$, the commutative part of $\mathcal{O}_A$, defined above, under the assumptions  {\tt (SH)}, {\tt (UCF)} and {\tt (FS)} for $A$.
This may allow us to call $\gamma _A$ the `canonical' ucp map (compare it with the canonical cp map  in \eqref{eq:canonicalcpsubshift} for subshifts with finite symbols).

Let $X_A$ be the character space of $\mathcal{D}_A$ and $\Phi \colon \mathcal{D}_A \to C(X_A)$ be the Gelfand--Naimark $*$-isomorphism.
That is, 
\[
X_A\coloneqq \{ \varphi : \mathcal D_A\to \mathbb C:\ \text{unital homomorphim}\},
\]
and for each $f\in\mathcal D_A$, $\Phi(f)$ is a $\mathbb{C}$-valued continuous function on $X_A$, given by
\[
\Phi(f)(\varphi)\coloneqq \varphi(f) \quad \text{for $\varphi\in X_A$.}
\]
We will extend $\sigma_A$ to $X_A$ and check $(\sigma_A)_* \circ \Phi = \Phi \circ \gamma_A$. 
First we aim to write elements in $X_A$ in a more concrete term.
Consider the set
\[
Y' _A \coloneqq  \left\{ \alpha =\{\alpha _j\}_{j=1}^m \in \Sigma^* _A \backslash \{ e \} : \alpha_{m} \ \text{is an $A$-infinite emitter}\right\}
\]
(recall that $e$ is the empty word).
We denote by $\{ \delta_\alpha \}_{\alpha \in Y' _A \cup \{e\}}$ the canonical orthonormal basis for $l^2(Y' _A \cup \{e\})$.
We define the sequence of partial isometries $\{ T_j \}_{j \in \N}$ on $l^2(Y' _A \cup \{e\})$ by
\[
T_j\delta_{\alpha}=
\begin{cases}
\delta_{j\alpha} &\text{if }A(j,\alpha_1)=1,\\
0 & \text{otherwise};
\end{cases}
\]
for $\alpha =\{\alpha _k\}_{k=1}^m \in Y' _A$ and
\[
T_j\delta_{e}=
\begin{cases}
\delta_{j} &\text{if $j$ is an $A$-infinite emitter},\\
0 & \text{otherwise}.
\end{cases}
\]
A straightforward calculation shows that
\[
T_j^*\delta_{\alpha}=
\begin{cases}
\delta_{\sigma_A\alpha} &\text{if } \alpha_1=j,\\
0 & \text{otherwise};
\end{cases}
\]
for any $\alpha =\{\alpha _k\}_{k=1}^m \in Y' _A$ and
\[
T_j^*\delta_{e}=0
\]
for any $j\in\mathbb{N}$.
One has 
\[
T^* _jT_j T^* _i T_i \delta_\alpha = T^* _iT_i T^* _j T_j \delta_\alpha, \quad
T^* _iT_i T_j \delta_\alpha = A(i, j) T_j \delta_\alpha, \quad
T^* _iT_j \delta_\alpha = \delta_{i, j} A(i, \alpha_1)\delta_\alpha
\]
and
\[
\left(\prod_{i \in X}T^* _iT_i\right)\left(\prod_{j \in Y} \left(1 - T^* _j T_j\right)\right)\delta_\alpha = \sum_{k=1}^\infty A(X, Y, k)T_kT_k^* \delta_\alpha
\]
for any $i, j \in \N$, $\alpha \in Y' _A$ and any finite subsets $X, Y \subset \N$ with $\lvert\{ k \in \N : A(X, Y, k) \neq 0\}\rvert < \infty$.
Moreover, one can show
\[
T^* _jT_j T^* _i T_i \delta_e = T^* _iT_i T^* _j T_j \delta_e, \quad
T^* _iT_i T_j \delta_e = A(i, j) T_j \delta_e
\]
and
\[
T^* _iT_j \delta_e = \delta_{i, j} \delta_e
\]
if $i$ and $j$ are $A$-infinite emitters.
Hence, to obtain a representation of $\mathcal{O}_A$ on $l^2(Y' _A \cup \{e\})$, it is sufficient to show that
\begin{equation}\label{EL4forT}
\left(\prod_{i \in X}T^* _iT_i\right)\left(\prod_{j \in Y} \left(1 - T^* _j T_j\right)\right)\delta_e = \sum_{k=1}^\infty A(X, Y, k)T_kT^*_k \delta_e
\end{equation}
for any finite subsets $X, Y \subset \N$ with $\lvert\{ k \in \N : A(X, Y, k) \neq 0\}\rvert < \infty$.

\begin{lem}\label{lem:rep}
Assume that  {\tt (FS)} holds for $A$.
Then one has a representation $\pi'$ of $\mathcal{O}_A$ on $l^2(Y' _A \cup \{e\})$ with 
\[
\pi'(S_j) = T_j
\]
for any $j \in \N$.
\end{lem}
\begin{proof}
Assume that {\tt (FS)} holds for $A$.
We check the equation (\ref{EL4forT}) for any finite subsets $X, Y \subset \N$ with $\lvert\{ k \in \N : A(X, Y, k) \neq 0\}\rvert < \infty$.
Suppose that any $i \in X$ is an $A$-infinite emitter and any $j \in Y$ is an $A$-finite emitter.
Then 
\[
A(X, Y, k) = \prod_{i \in X} A(i, k) \prod_{j \in Y} \left(1 -A(j, k)\right) = 1
\] 
holds for infinitely many $k \in \N$ since  any $A$-finite emitter cannot be an $A$-infinite emitter.
Thus, the condition {\tt (FS)} allows us to assume that $X$ contains an $A$-finite emitter or  $Y$ contains an $A$-infinite emitter.
If $X$ contains an $A$-finite emitter, then one has
\[
\left(\prod_{i \in X}T^* _iT_i\right)\delta_e = 0.
\]   
In addition, if $Y$ contains an $A$-infinite emitter, then one has
\[
\left(\prod_{i \in Y}(1 - T^* _iT_i)\right)\delta_e = 0.
\]
Consequently, we proved that the left-hand side of the equation (\ref{EL4forT}) equals to $0$.
Moreover, since $T_kT^* _k\delta_e =0$ for any $k \in \N$, the right hand side of the equation (\ref{EL4forT}) equals to $0$.
\end{proof}

For each $\alpha \in Y' _A \cup \{ e \}$, we define a state $\tau_\alpha$ on $\mathcal{O}_A$ by
\[
\tau_\alpha(f) \coloneqq  \langle \pi'(f) \delta_\alpha, \delta_\alpha \rangle \quad \text{for } f \in \mathcal{O}_A.
\]
Note that $\tau_\alpha(S_u S^* _u)$ equals to either $1$ or $0$ for any finite admissible word $u$. 
Thus, the multiplicative domain of $\tau_\alpha$ includes $\mathcal{D}_A$ and therefore $\tau_\alpha \vert_{\mathcal{D}_A}$ is a (nonzero) character on $\mathcal{D}_A$ (see proposition 1.5.7 of \cite{brown2008textrm} for the multiplicative domain).
In summary, we have
\[
Y_A \coloneqq  \{ \tau_\alpha \vert_{\mathcal{D}_A} \}_{\alpha \in Y' _A \cup \{e\}} \subset X_A.
\]  
Recall that $\pi$ is the canonical representation of $\mathcal{O}_A$ on $l^2(\Sigma_A)$ (recall the footnote \ref{footnote_3}) defined similarly to $\pi'$.
For each $x \in \Sigma_A$, we define
\[
\tau_x(f) \coloneqq  \langle \pi(f) \delta_x, \delta_x \rangle \quad \text{for }f \in \mathcal{O}_A.
\]
A continuous embedding of $\Sigma_A$ into $X_A$ is given by the restrictions on $\mathcal{D}_A$ of $\tau_x$'s (see Definition 4.40 in \cite{Raszeja}).
We write this embedding
\begin{equation}\label{eq:1228a}
\iota \colon \Sigma_A \to X_A.
\end{equation}

\begin{prop}
Assume that {\tt (FS)} holds for $A$.
Then one has $Y_A = X_A \setminus \iota(\Sigma_A)$.
\end{prop}
\begin{proof}
First, we show $\iota(\Sigma_A) \cap Y_A = \emptyset$.
Take arbitrary $x \in \Sigma_A$ and $\alpha \in Y' _A \cup \{ e \}$.
Then for any $m \geq \lvert\alpha\rvert+1$, there exists an admissible word $\mu$ with $\lvert\mu \rvert =m$ and $\tau_x(S_\mu S^* _\mu) =1$.
On the other hand, one has $\tau_\alpha(S_\mu S^* _\mu)=0$ since $\lvert\mu\rvert \geq \lvert\alpha\rvert +1$.
Hence, $\tau_x \neq \tau_\alpha$.
Since $x$ and $\alpha$ are arbitrary, we have $\iota(\Sigma_A) \cap Y_A = \emptyset$ and hence $Y_A \subset X_A \setminus \iota(\Sigma_A)$.

Next, we prove $X_A \setminus \iota(\Sigma_A) \subset Y_A$.
We aim to apply Proposition 4.63 and Proposition 4.64 in \cite{Raszeja}.
We observe
\[
\left\{ i \in \N : \tau_e(S^* _iS_i)=1\right \} = \left\{ i \in \N : i \text{ is an $A$-infinite emitter} \right\}. 
\]
Then we conclude that $\tau_e$ is the unique character corresponding to the empty word and the set of all $A$-infinite emitters (see Proposition 4.63 in \cite{Raszeja}).
Moreover, since we assume {\tt (FS)}, the sequence of columns of $A$ converges to $a = \{a_i\}_{i \in \N} \in \{0, 1 \}^\N$ defined to be
\[
a_i = \begin{cases}
1 & \text{if $i$ is an $A$-infinite emitter,} \\
0 & \text{otherwise}.
\end{cases}
\]  
Here we consider the product topology of discrete spaces on $\{0, 1\}^\N$.
Hence Proposition 4.64 \cite{Raszeja} shows that $\tau_e$ is the unique character whose stem is the empty word (see \cite{Raszeja} or the original paper \cite{EL1999} for the definition of stem).

Similarly, one can check that
\[
\left\{ i \in \N : \tau_\alpha(S_\alpha S^* _iS_i S^* _\alpha)=1 \right\} = \left\{ i \in \N : i \text{ is an $A$-infinite emitter} \right\}
\] 
for any $\alpha \in Y' _A$.
This shows that $\tau_\alpha$ is the unique character corresponding to $\alpha$ and the set of all $A$-infinite emitters in the sense of Proposition 4.63 in \cite{Raszeja}.
In addition, for any $\alpha \in Y' _A$,
Proposition 4.64 in \cite{Raszeja} and the uniqueness of the accumulation point, which we observed above, implies that each $\tau_\alpha$ is the unique character whose stem is $\alpha$.
This together with the above observation for $e$ shows $X_A \setminus \iota(\Sigma_A) \subset Y_A$ since any element in $\Sigma_A$ has a stem whose length is not finite.
\end{proof}

Now, we extend the shift map $\sigma_A$ on $\Sigma_A$ to $\hat{\sigma}_A$ on $X_A$ by defining
\[
\hat{\sigma}_A(\tau_e) \coloneqq  \tau_e
\quad\text{and}\quad
\hat{\sigma}_A(\tau_\alpha) \coloneqq  \tau_{\sigma_A(\alpha)}
\quad\text{for $\alpha \in Y' _A$},
\]
where $\sigma_A(\alpha)$ is given by the usual left shift operation on the set of nonempty finite admissible words.
We also define
\[
\hat{\sigma}_A(\tau_x) \coloneqq  \tau_{\sigma_A(x)} \quad \text{for }x \in \Sigma _A.
\] 
By definition, we have
\[
\hat{\sigma}_A \circ \iota = \iota \circ \sigma_A.
\]

For the following proposition, we recall that $\Phi(f)$ is a (weak*) continuous function on $X_A$ for any $f \in \mathcal{D}_A$. 
\begin{prop}\label{commprop}
Under the assumptions {\tt (UCF)} and {\tt (FS)},  one has 
\[
\left(\Phi(f)\right)\left(\hat{\sigma}_A(\tau)\right) = \left(\hat{\sigma}_A(\tau)\right)(f) = \tau\left(\gamma_A(f)\right) = \left((\Phi \circ \gamma_A)(f)\right)(\tau)
\]
for any $f \in \mathcal{D}_A$ and $\tau \in X_A$. 
\end{prop}
\begin{proof}
Since $\Phi$ gives the Gelfand--Naimark representation of $\mathcal{D}_A$, one has 
\[
\left(\Phi(f)\right)(\hat{\sigma}_A(\tau)) = \left(\hat{\sigma}_A(\tau)\right)(f) \quad \text{and} \quad \tau\left(\gamma_A(f)\right) = \left((\Phi \circ \gamma_A)(f)\right)(\tau)
\]
for any $f \in \mathcal{D}_A$ and $\tau \in X_A$. 
Moreover, the commutative C$^*$-algebra $\mathcal{D}_A$ is generated by $\{ S_\mu S^* _\mu \}_{\mu \in \Sigma^* _A}$. 
Thus, it suffices to show
\begin{equation}\label{commeq}
\left(\hat{\sigma}_A(\tau)\right)\left(S_\mu S^* _\mu\right) = \tau\left(\gamma_A\left(S_\mu S^* _\mu\right)\right)
\end{equation}
for any $\mu \in \Sigma_A ^*$ and $\tau \in X_A$.
It is clear that the equation (\ref{commeq}) holds for $\mu =e$ since $S_e S^* _e =1$.
We check the equation (\ref{commeq}) for $\mu \in \Sigma_A ^* \backslash \{e \}$. 

For $\tau = \tau_e$, one has
\[
\left(\hat{\sigma}_A(\tau_e)\right)\left(S_\mu S^* _\mu\right) = \tau_e\left(S_\mu S^* _\mu\right) =0
\]
and
\[
\tau_e\left(\gamma_A\left(S_\mu S^* _\mu\right)\right) = \sum_{i =1}^{\infty}\tau_e\left(S_i S_\mu S^* _\mu S^* _i\right) = 0. 
\]
Thus, the equation (\ref{commeq}) holds for $\tau = \tau_e$.

Take $x \in \Sigma_A$.  
For $\tau = \tau_x$, one computes
\begin{equation*}
\left(\hat{\sigma}_A(\tau_x)\right)\left(S_\mu S^* _\mu\right)
= \tau_{\sigma_A(x)}\left(S_\mu S^* _\mu\right) 
=\begin{cases}
1 & \text{if } A(x_1, \mu_1) = 1 \text{ and } \sigma_A(x) \in [\mu], \\
0 & \text{otherwise};
\end{cases}
\end{equation*}
and
\begin{equation*}
\begin{split}
\tau_x\left(\gamma_A\left(S_\mu S^* _\mu \right)\right)
&= \sum_{i =1}^{\infty}\left\langle \pi\left(S_i S_\mu S^* _\mu S^* _i\right) \delta_x, \delta_x \right\rangle \\  
&=\begin{cases}
1 & \text{if } A(x_1, \mu_1) = 1 \text{ and } \sigma_A(x) \in [\mu], \\
0 & \text{otherwise}.
\end{cases}
\end{split}
\end{equation*}
Thus, the equation (\ref{commeq}) holds for $\tau = \tau_x$.

Take $\alpha \in \Sigma^* _A$.
For $\tau = \tau_\alpha$, one computes
\begin{equation*}
\begin{split}
\left(\hat{\sigma}_A(\tau_\alpha)\right)\left(S_\mu S^* _\mu\right)
&= \tau_{\sigma_A(\alpha)}\left(S_\mu S^* _\mu\right) \\
&=\begin{cases}
1 & \text{if } \lvert\alpha\rvert \leq \lvert\mu \rvert + 1, A(\alpha_1, \mu_1) = 1 \text{ and } \sigma_A(\alpha) = \mu \nu \text{ for some } \nu \in \Sigma^* _A , \\
0 & \text{otherwise};
\end{cases}
\end{split}
\end{equation*}
and
\begin{equation*}
\begin{split}
\tau_{\alpha}\left(\gamma_A(S_\mu S^* _\mu )\right)
&= \sum_{i = 1}^{\infty}\left<\pi'(S_i S_\mu S^* _\mu S^* _i)\delta_{\alpha}, \delta_{\alpha}\right>\\  
&=\begin{cases}
1 & \text{if } |\alpha| \leq |\mu | + 1, A(\alpha_1, \mu_1) = 1 \text{ and } \sigma_A(\alpha) = \mu \nu \text{ for some } \nu \in \Sigma^* _A, \\
0 & \text{otherwise}.
\end{cases}
\end{split}
\end{equation*}
Thus, the equation (\ref{commeq}) holds for $\tau = \tau_{\alpha}$.
Hence, we finished the proof.
\end{proof}

\begin{prop}\label{prop:wconti}
Under the assumptions {\tt (UCF)} and {\tt (FS)}, the extended $\hat{\sigma}_A$ on $X_A$ is weak* continuous.
\end{prop}

\begin{proof}
Take any $\tau \in X_A$ and any net $\tau^{(\lambda)} \in X_A$ converging to $\tau$ with respect to the weak* topology.
Apply Proposition \ref{commprop} to any $f \in \mathcal{D}_A$, then we have
\[
\left\lvert \left(\hat{\sigma}_A \left(\tau^{(\lambda)}\right)\right) (f) - \left(\hat{\sigma}_A (\tau)\right) (f) \right\rvert
= \left\lvert \tau^{(\lambda)}\left(\gamma_A(f)\right) - \tau\left(\gamma_A(f)\right) \right\rvert.
\]
Since $\gamma_A(\mathcal{D}_A) \subset \mathcal{D}_A$ and $\tau^{(\lambda)} \in X_A$ converges to $\tau$ with respect to the weak* topology, we have 
\[
\lim_{\lambda} \left\lvert \tau^{(\lambda)}\left(\gamma_A(f)\right) - \tau\left(\gamma_A(f)\right) \right\rvert =0
\]
for any $f \in \mathcal{D}_A$.
Combining these,
\[
\lim_{\lambda}\left\lvert \left(\hat{\sigma}_A \left(\tau^{(\lambda)}\right)\right) (f) - \left(\hat{\sigma}_A (\tau)\right) (f) \right\rvert = 0.
\]
This implies that $\hat{\sigma}_A(\tau^{(\lambda)})$ converges to $\hat{\sigma}_A (\tau)$.  
\end{proof}

By virtue of Proposition \ref{prop:wconti}, one can define the pullback operation $(\hat{\sigma}_A)_* : f\mapsto f\circ \hat{\sigma} _A$ by $\hat{\sigma} _A$ on $C(X_A)$.
Using $\iota_*$ defined by $f\mapsto f\circ \iota$ for $f\in C(X_A)$ as well, we have the following commutative diagram.

\[
  \begin{CD}
     C(\Sigma _A) @>{(\sigma_A)_*}>> C(\Sigma _A)\\
  @A{\iota_*}AA    @A{\iota_*}AA \\
     C(X_A)   @>{(\hat \sigma_A)_*}>>  C(X_A) \\
  @A{\Phi}AA	@A{\Phi}AA\\
     \mathcal{D}_A @>{\gamma_A}>> \mathcal{D}_A\\
     @A{id\vert_{\mathcal{D}_A}}AA	@A{id\vert_{\mathcal{D}_A}}AA\\
     \mathcal{O}_A @>{\gamma_A}>> \mathcal{O}_A\\
  \end{CD}
\]

\section{Main result}\label{s:main}

As the final preparation before stating our main results, we precisely define the rest of the assumptions in Theorem \ref{thm:main}.
Let $\rho _A^{-1} (B)$ be the inverse map of the multi-valued map on $\N$ induced by the arrows in $A$, that is, $\rho _A^{-1}: \N \to 2^{\N}$ such that $\rho _A^{-1}(i) \ni j$ if and only if $A(j,i)=1$ for each $i\in \N$. 
Write $\rho _A ^{-1}(B)\coloneqq \bigcup _{i\in B} \rho _A^{-1}(i)$  for $B\subset \N$.
For $m, n\in \N$, we inductively define $N(m,n)$ by
\[
N(1,n)\coloneqq \{1,2,\ldots ,n\}, \quad N(m+1,n)\coloneqq \rho _A^{-1}(N(m,n)).
\]
\begin{dfn}\label{asympfin}
We say that $A$ is \textit{asymptotically finite} {\tt (AF)}  if there exists a natural number $n_0$ such that $N(1,n) \subset N(2,n)$ for any $n \geq n_0$.
\end{dfn}
\begin{lem}\label{seqofsets}
Assume that   {\tt (AF)} holds for $A$ with  $n_0 \in \N$.
Then, for any $m\in \N$ and $n \geq n_0$, 
\[
N(m, n) \subset N(m+1, n).
\]
\end{lem}

\begin{proof}
Fix $n\ge n_0$.
We show the lemma by induction for $m \in \N$.
The case $m=1$ is exactly {\tt (AF)}.
Assume that $N(m, n) \subset N(m+1, n)$.
Take $i \in N(m+1, n)=\rho _A^{-1}(N(m,n))$.
Then, $A(i, j)=1$ for some $j \in N(m, n) \subset N(m+1, n)$.
This implies $i \in N(m+2, n)$. 
\end{proof}
Assume that {\tt (UCF)} holds.
Then, for any $m, n \in \N$, $N(m,n)$ is a finite set, and
\begin{align*}
I(m,n)\coloneqq  \max N(m,n)
\end{align*}
is well-defined.
Assume also that {\tt (AF)} holds. 
Then, $n\le I(m,n)\le I(m+1,n)$ for $m\in\N$, $n\ge n_0$.

\begin{dfn}\label{dfn:SI}
We say that the \emph{subexponential inverse image property} {\tt (SI)} holds for $A$ if
 \[
\mathcal I \coloneqq \limsup _{n\to\infty} \limsup _{m\to \infty} \frac{\log I(m,n)}{m} =0.
\]
\end{dfn}

For each $m,n,p\in \N$ with $p\le m-1$, we define
\[
L(m, n) \coloneqq  \left\{  \{\alpha_i\}_{i=1}^m \in \Sigma _A^*:  \alpha _m \leq n \right\}, \quad 
L_p(m,n)\coloneqq \left\{ \{ \alpha_i\}_{i=1}^p :  \{ \alpha_i\}_{i=1}^m\in L(m,n)\right\}.
\]
Then, it is immediate to see that $L_1(m,n) = N(m,n)$.
Note that 
\[
L(m,n) \subset L_p(m,n)L(m-p,n) \coloneqq  \left\{  \{ \alpha_i\}_{i=1}^m : \{ \alpha_i\}_{i=1}^p\in L_p(m,n), \;  \{ \alpha_i\}_{i=p+1}^m\in L(m-p,n)\right\}
\]
while the inverse inclusion is not true in general.
However, it will be important to estimate $\vert L_p(m,n)\vert $ from above by using $\frac{\vert L(m,n)\vert }{\vert L(m-p,n)\vert }$.
Taking it into account, we introduce the following condition.
\begin{dfn}
We say that the \emph{subexponential distortion property} {\tt (SD)} holds for $A$
 \[
\mathcal D\coloneqq \limsup _{n\to\infty} \limsup _{m\to \infty} \frac{\log D(m,n)}{m} =0,
\]
where
\[
D(m,n)\coloneqq \max _{1\le p\le m-1} \frac{\vert  L_p(m,n)\vert \cdot \vert L(m-p,n)\vert }{\vert L(m,n)\vert }.
\]
\end{dfn}
\begin{remark}\label{rem:1117}
Assume that $A$ has the following self-similar inverse image structure: for $m,n,p\in \N$ with $p\le m-1$ and $i, j \in N(m-p+1,n)$, 
\begin{equation}\label{eq:1211a}
\vert L^i_p(m,n) \vert = \vert L^j_p(m,n)\vert ,
\end{equation}
where $L^j_p(m,n)\coloneqq \{ \{\alpha _i\}_{i=1}^p \in L_p(m,n) : \alpha _p = j\}$.
(Note that the renewal shift satisfies the condition with $\vert L^i_p(m,n)\vert = 2^{p-1}$.)
Then, it is straightforward to see that\footnote{For each $\{\alpha _i\}_{i=1}^m\in L(m,n)$ and $j \in N(m-p+1,n)$, denote by $\Gamma (\{\alpha _i\}_{i=1}^m, j ) $
the set of sequences $\{ \gamma _i\}_{i=1}^m$ such that $\gamma _i= \alpha _i$ for $i=p+1,\ldots ,m$, $\gamma _{p} = j$ and $\{ \gamma _i\} _{i=1}^p\in L_p(m,n)$.
Fix $j \in N(m-p+1,n)$.
Then, it follows from \eqref{eq:1211a} that $\vert \Gamma (\{\alpha _i\}_{i=1}^m,j ) \vert = \vert L_p^j(m,n)\vert $, and that for each $\{\beta_i\} _{i=p+1}^m \in L(m-p,n) $, there are exactly $\vert L^{j}(m,n)\vert $-many $\{\alpha _i\} _{i=1}^m$'s in $L(m,n)$ such that 
\[
L_p^j(m,n) \times \{ \{\beta_i\} _{i=p+1}^m\} = \Gamma (\{\alpha _i\} _{i=1}^m, j).
\]
Hence, 
\[
\vert L_p^j(m,n) L(m-p,n) \vert \le \frac{1 }{\vert L^{j}(m,n)\vert }\sum _{\{\alpha _i\}_{i=1}^m\in L(m,n)}  \vert \Gamma (\{\alpha _i\}_{i=1}^m, j )\vert  = \vert L(m,n)\vert .
\]
Thus, we immediately get the claim by \eqref{eq:1211a}.}
\[
\frac{\vert  L_p(m,n)\vert \cdot \vert L(m-p,n)\vert }{\vert L(m,n)\vert } \le I(m-p+1,n).
\]
Hence, $\mathcal D\le \mathcal I$ when {\tt (AF)} holds.
That is, {\tt (AF)} and {\tt (SI)} imply {\tt (SD)} under \eqref{eq:1211a}.
\end{remark}

\begin{dfn}\label{procolfin}
A finite division $H^{(1)}, \cdots, H^{(m)}$ of $\N$ is said to be an \textit{$A$-orthogonal division} if $\vert \rho _A^{-1}(k) \cap H^{(l)}\vert \le 1$ for any $k \in \N$ and $l \in \{1, \cdots, m\}$.
We say that $A$ is \textit{properly column finite} {\tt (PCF)} if there exists an $A$-orthogonal division. 
\end{dfn}
\begin{remark}\label{rem:1117}
Let $H^{(1)}, \cdots, H^{(m)}$ be an $A$-orthogonal division.
Then for each $ i, j \in H^{(l)}$ with $l \in \{1, \cdots, m \}$,  one has 
\[
S_iS^* _j =0 \quad \text{if and only if $i\neq j$}.
\]
Thus,
\[
S_{(l)} \coloneqq \sum_{i \in H^{(l)}} S_i \quad (l=1, \cdots, m)
\]
are well-defined partial isometries (the sum is given by strong operator topology on some appropriate $*$-representations of $\mathcal O_A$),
and $S_{(l)}S_{(l)}^* = \sum _{i, j \in H^{(l)}} S_i S_j^* = \sum _{i \in H^{(l)}} S_i S_i^* $.
Hence, it follows from \eqref{eq:1111d} that 
\[
\sum_{l=1}^mS_{(l)}S_{(l)}^* =1.
\]
Although we will not use the equality in this form, this well explains why {\tt (PCF)} is useful in our proof, where the argument of \cite{BG2000} in the finite matrix case is heavily used.
\end{remark}

\begin{lem}
{\tt (PCF)} implies {\tt (UCF)}.
\end{lem}
 
\begin{proof}
By assumption, we have an $A$-orthogonal division $H^{(1)}, \cdots, H^{(m)}$.
For each $i \in \N, l \in \{1, \cdots, m\}$, there exists at most one $j \in H^{(l)}$ with $A(i, j) =1$.
Thus, we have $|C_i| \leq m$ for any $i \in \N$.
\end{proof}

When $A$ satisfies {\tt (PCF)}, one can find the smallest natural number $m_A$ such that there exists an $A$-orthogonal division consisting of $m_A$ divisions. 
An $A$-orthogonal division is said to be \textit{minimal} if it consists of $m_A$ divisions. 
For each minimal $A$-orthogonal division $H^{(1)}, \cdots, H^{(m_A)}$,
we define 
\[
P_A(H^{(1)}, \cdots, H^{(m_A)}) \coloneqq   |\{l \in \{1, \cdots, m_A\} : H^{(l)} \ \text{is an infinite set} \}|.
\] 
We write
\[
P_A \coloneqq  \min \{ P_A(H^{(1)}, \cdots, H^{(m_A)}) : H^{(1)}, \cdots, H^{(m_A)} \ \text{is a minimal $A$ -othogonal division} \}. 
\]
\begin{dfn}\label{procolfin}
We say that $A$ is \textit{orthogonal} {\tt (O)} if {\tt (PCF)} holds and $P_A=1$.
\end{dfn}

Now we are ready to state our main results precisely.
\begin{thm} \label{thm:1}
Assume that {\tt (SH)}, {\tt (FS)}, {\tt (AF)} and {\tt (PCF)} hold for $A$ and there exists a KMS state for the canonical gauge action $\Gamma$.
Then, we have
\[
ht (\gamma _A) \leq  h_{G}(\sigma _A) + \log P_A +\mathcal I 
\]
and
\begin{equation*}
ht (\gamma _A) \leq \max \left\{h_{G}(\sigma _A) , \log P_A\right\} +\mathcal I +\mathcal D.
\end{equation*}
In particular, if in addition  $\log P_A \le h_{G}(\sigma _A)$ (such as  {\tt (O)} holds) and  {\tt (SI)} and {\tt (SD)}  hold, then
\begin{equation*}
ht (\gamma _A) \leq h_{G}(\sigma _A).
\end{equation*}
\end{thm}

 \begin{thm} \label{thm:2}
Assume that {\tt (SH)}, {\tt (UCF)} and {\tt (FS)} hold for $A$. 
Then, it holds that
\begin{equation*}
 h_{G}(\sigma _A)\leq ht (\gamma _A).
\end{equation*}
\end{thm}

Note that Theorem \ref{thm:main} immediately follows from Theorems \ref{thm:1} and \ref{thm:2}.

\begin{example}\label{ex:1227}
There exists a (uniformly) locally finite  $A$ with $\mathcal I >0 $.
For example, consider the $0$-$1$ matrix $A$ given by  
\begin{equation}\label{eq:counter_example}
A(i,j) =
\begin{cases}
1 \quad &  (\text{$i\in \{ 2j-1, 2j\}$})\\
0 \quad & (\text{otherwise})
\end{cases}.
\end{equation}
Then we have $I(m,n) = 2^{m-1}\cdot n$,  so that $\mathcal I(n_1) = \log 2$ for any $n_1\in \mathbb N$.

There also exists a $0$-$1$ matrix $A$ with $h_G(\sigma_A) < \log P_A$.
For example, consider the $0$-$1$ matrix $A$ given by  
\begin{equation}\label{eq:counter_example}
A(i,j) =
\begin{cases}
1 \quad &  (\text{$i=1$ or $j\in \{ i-1, i, i+2, i+3, i+5, i+6\}$})\\
0 \quad & (\text{otherwise}).
\end{cases}
\end{equation}
Then we have $P_A = 8$ since we need 8 $A$-orthogonal divisions to separate 8 letters $2, 3, 4, 5, 6, 7, 8, 9 \in \N$.
On the other hand, we have $M_A =7$.
Furthermore, as a general fact, it holds that $h_G(\sigma_A) \le \log M_A$ since $h_G(\sigma _A)= \lim _{m\to \infty} \frac{1}{m}\log A^{m}(j,j)$ for any $j\in \N$ (cf.~\cite{Sariglecturenote}; refer also to Lemma \ref{lem:la} and \eqref{eq:1221a}).
In summary, we have
\[
h_G(\sigma_A) \leq \log M_A < \log P_A.
\]
\end{example}

\begin{example}\label{example:1}
The renewal shift (given in \eqref{eq:renewal}) satisfies {\tt (SH)}, {\tt (FS)}, {\tt (AF)},  {\tt (SI)}, {\tt (SD)} and {\tt (O)} and admits a KMS state for the canonical gauge action $\Gamma$:
\begin{itemize}
\item  {\tt (SH)} holds because for any $i, j\in \N$, $A^{i}(i,j)=A(i,i-1)\cdots A(2,1)A(1,j) =1$;
 \item  {\tt (FS)} holds with an $A$-infinite emitter $1$ and $A$-finite emitters $2, 3, \ldots$;
  \item {\tt (AF)} and {\tt (SI)} hold with $N(m,n)=\{1,\ldots ,m+n\}$;
  \item {\tt (SD)} holds as mentioned in Remark \ref{rem:1117};
 \item {\tt (O)} holds with $m_A=2$ and $H^{(1)} =\{1\}$, $H^{(2)} =\{2, 3,\ldots \}$;
 \item A KMS state (with the inverse temperature parameter $\beta = \log 2$) for $\Gamma$ exists due to Proposition 4.117, Remark 5.15 and Corollary 5.29 of \cite{Raszeja}.
 \end{itemize}
Hence, it follows from Theorems \ref{thm:1} and \ref{thm:2} that   $ht(\gamma _A) = h_G(\sigma _A) $ holds.
\end{example}

\begin{example}\label{example:2}
Another example to which Theorems  \ref{thm:1} and \ref{thm:2} are applicable is the lazy renewal shift given in \eqref{eq:lazyrenewal}.
In fact, {\tt (SH)}, {\tt (FS)}, {\tt (AF)} and {\tt (SI)} hold for the same reason as  the renewal shift.
Similarly, {\tt (PCF)} with $P_A =2$ holds, by taking $m_A=3$, $H^{(1)} =\{1\}$, $H^{(2)} =\{2, 4,\ldots \}$ and $H^{(3)} =\{3, 5,\ldots \}$.
To see {\tt (SD)}, for each $m,n,p,q\in \mathbb N$ with $q\le p \le m$, we let 
\[
a_{q,p}(m,n)\coloneqq \lv \left\{\{\alpha _i\}_{i=q}^p : \{\alpha _i\}_{i=1}^m \in L(m,n), \; \alpha _ q=1\right\}\rv
\]
and
\[
b_{q,p}(m,n)\coloneqq \lv \left\{\{\alpha _i\}_{i=q}^p : \{\alpha _i\}_{i=1}^m \in L(m,n), \; \alpha _ q\neq 1\right\}\rv.
\]
Then, with $c_{q,p}(m,n) \coloneqq  a_{q,p}(m,n) + b_{q,p}(m,n)$, it holds that
\[
\vert L(m,n)\vert =c_{1,m}(m,n), \quad 
\vert L(m-p,n)\vert =c_{p+1,m}(m,n), \quad
  \vert L_p(m,n)\vert =c_{1,p}(m,n) 
\]
and
\[
\left(
\begin{array}{c}
a_{q-1,p}(m,n)\\
b_{q-1,p}(m,n)
\end{array}
\right)= B
\left( \begin{array}{c}
a_{q,p}(m,n)\\
b_{q,p}(m,n)
\end{array}\right)
 \quad \text{with} \; 
B=
\left(\begin{array}{cc}
1 & 1\\
1 & 2
\end{array}
\right) .
\]
Furthermore, it holds that
\[
\left( \begin{array}{c}
a_{p,p}(m,n)\\
b_{p,p}(m,n)
\end{array}\right) = 
\left( \begin{array}{c}
1\\
n + m- p -1
\end{array}\right) 
\]
Hence, a straightforward calculation (starting from the diagonalization of $B$) gives that, by denoting $\lambda _+\coloneqq  \frac{3+\sqrt 5}{2}$ and $\lambda _-\coloneqq  \frac{3-\sqrt 5}{2}$,
\begin{equation}\label{eq:1212a}
\vert L(m,n)\vert = (C_1n +D_1)\lambda_+^{m-1} +(C_2n+D_2) \lambda _-^{m-1} 
\end{equation}
with constants $C_1, C_2 >0$ and $D_1$ and $D_2$ being independent of $n$ and $m$, and similar estimates hold for $\vert L_p(m,n)\vert$ and $\vert L(m-p,n)\vert $.
Thus, we get
\[
\frac{\vert L_p(m,n)\vert \cdot \vert L(m-p,n)\vert }{\vert L(m,n)\vert } \le C (n+m -p) +D
\]
with constants $C>0$ and $D$ being independent of $n, m$ and $p$.
This immediately concludes {\tt (SD)}.
On the other hand, by \eqref{eq:1212a}, one can easily see that $h_G(\sigma _A)= \log \lambda _+ $ (see Proposition \eqref{lem:la1109} for details), so that $h_G(\sigma _A) >\log P_A$.

The existence of a KMS state can be seen by an observation similar to one for the renewal shift in Example \ref{example:1}, as follows.
 Recall that $[1]=\{\{x_i\}_{i=1}^\infty\in\Sigma_A : x_1=1\}$.
 For $\beta=\log \lambda_+$,
\begin{align*}
\sum_{n=1}^{\infty} e^{-n \left(h_G(\sigma_A)-\beta\right)} e^{-\beta n}\lv\{ x\in[1]:\sigma^nx=x \} \rv
&=\sum_{n=1}^{\infty} e^{-n h_G(\sigma_A)} \lv\{ x\in[1]:\sigma^nx=x \} \rv\\
&=\sum_{n=1}^{\infty} \lambda_+^{-n} \lv L(n,2)\rv
=\infty.
\end{align*}
This shows that the lazy renewal shift (with the potential $-\beta$) is recurrent (moreover, one can show that this is actually positive recurrent).
Thus generalized Ruelle's Perron--Frobenius theorem implies that there is some $e^{\beta}$-conformal measure $\nu$ on $\Sigma_A$, which is finite on each cylinder set, and we only need to show that $\nu$ is a finite measure.
(For the definition of $e^{\beta}$-conformal measures and the connection between it and a KMS state, see Definition 1.21, Corollary 1.39, Theorem 5.13 and Remark 5.15 in \cite{Raszeja}.)
Recall that an $e^{\beta}$-conformal measure corresponds to an eigenmeasure of the Ruelle operator with the eigenvalue $e^{h_G(\sigma_A)-\beta}=1$ and one has
\[
\nu([1])=\int e^{-\beta}\sum_{y:\sigma_Ay=x}1_{[1]}(y)d\nu(x)
=\lambda_+^{-1}\sum_{s:A(1,s)=1}\nu([s])
=\lambda_+^{-1}\nu\left(\Sigma_A\right).
\]
Therefore, $\nu$ is a finite measure since $\nu([1])<\infty$ and there exists a KMS state.

Consequently, by virtue of  Theorems \ref{thm:1} and \ref{thm:2},  we get the optimal estimate
 \[
  ht(\gamma _A) = h_G(\sigma _A) 
 \]
 for the lazy renewal shift $A$.
\end{example}

\begin{remark}
The above approach for the existence of a KMS state, passing through an $e^{\beta}$-conformal measure, seems to apply not only to renewal type shifts but to more general Markov shifts, under appropriate assumptions such as {\tt (SH)} and {\tt (FS)}.
This would be established in the forthcoming paper.
\end{remark}


\section{Problems}\label{s:problems}

 \subsection{Spectral radius}
When a $0$-$1$  matrix $A$ is a finite matrix, it is well known that there is a nice relation between $h_G(\sigma _A)$  and the spectral radius $r(A)$ of $A$, that is, $h_G(\sigma _A)=\log r(A)$.
However, the definition of the `spectral radius' $r(A)$ of a $0$-$1$  matrix $A=\{A(i,j)\}_{(i,j)\in D^2}$ with an infinite alphabet $D$   is far from trivial. 
In the context of graph theory, it is usual to assume that $A$ is uniformly locally finite and define $r(A)$ as the spectral radius of the linear operator $L_A: l^2(D)\to l^2(D)$ given by
\begin{equation}\label{dfn:la}
L_A\left(\alpha \right) \coloneqq  \left\{ \sum_{j=1}^\infty A(i, j)\alpha _j \right\}_{i\in\mathbb N} \qquad (\alpha =\left  \{\alpha _i\right\}_{i\in\mathbb N}   \in l^2(D)).
\end{equation}
In fact, $A$ is uniformly locally finite if and only if $L_A: l^2(D)\to l^2(D)$ is bounded, and moreover, it holds that
\begin{equation}\label{eq:1117g}
r(L_A: l^2(D)\to l^2(D)) = \sup \left\{ r(F) : \text{$F$: finite submatrix of $A$}\right\}
\end{equation}
(cf.~\cite{Mohar82}), where $r(S: E\to E)$ is the spectral radius of a bounded operator $S: E\to E$ on a Banach space $E$.
On the other hand, it is known that if $\sigma _A$ is mixing, then $h_G(\sigma _A)$ also coincides with the right-hand side of \eqref{eq:1117g}, refer to e.g.~\cite{Sariglecturenote}. 
Consequently, one gets 
\[
h_G(\sigma _A) = r(L_A: l^2(D)\to l^2(D)). 
\]

However, several important countable Markov shifts including the renewal shift violate the local finiteness.
Thus, according to the following lemma, we consider $L_A: l^1(D)\to l^1(D)$ by \eqref{dfn:la} with $l^1(D)$ instead  $l^2(D)$.
Recall Definition \ref{dfn:ucf} for $M_A$.
\begin{lem}\label{lem:la}
 {\tt (UCF)} holds for $A$  if and only if $L_A$ is bounded on $l^1(D)$.
In fact, 
\[
\Vert L_A\Vert _{l^1} =M_A.
\]
\end{lem}
\begin{proof}
Assume that $A$ satisfies {\tt (UCF)}. 
For any $\{\alpha_i\}_{i=1} ^\infty \in l^1(D)$, the series $\sum_{j\in D} A(i, j)\alpha_j$ converges absolutely because $ A(i,j)=0$ or $1$ for each $i, j\in\mathbb N$.
Furthermore, since $A(i,j)\vert \alpha _j\vert$ is non-negative for any $i, j\in \mathbb N$, we have
\[
\sum_{i\in D} \left| \sum_{j\in D} A(i, j)\alpha_j \right| \leq \sum_{j\in D} \sum_{i\in D}  A(i, j)|\alpha_j| \leq \sum_{j\in D} M_A |\alpha_j| 
= M_A \sum_{j\in D}|\alpha_j|.
\]
This shows that $L_A $ is a bounded operator on $l^1(D)$ and $\|L_A\| _{l^1}\leq M_A$.

We assume that $L_A: l^1(D)\to l^1(D)$ is bounded.
For each $ j \in D$,  let $(\delta^{(j)} _i)_{i\in D}$ be a sequence defined by
\[
\delta_i ^{(j)} =
\begin{cases}
1 \quad &  (\text{if $j=i$}) \\
0  \quad & (\text{otherwise})
\end{cases}.
\]
Then we have
\[
\left\vert \{ i \in \N :  A(i, j) =1\} \right\vert = \|L_A ((\delta^{(j)} _i)_{i\in D}) \| _{l^1} \leq \|L_A\|_{l^1}
\]
for any $j \in \N$.
Thus, $A$ is uniformly column finite and $M_A \leq \|L_A\|_{l^1}$.
This completes the proof.
\end{proof}
 
On the other hand, a slight modification of the proof of Theorem \ref{thm:1} tells us that, under the assumptions of Theorem \ref{thm:1}, it holds that\footnote{
It will be shown that $ht(\gamma _A)\le \sup _{n\ge N} \limsup _{m\to\infty} \frac{1}{m}\log (\sum_{p=0}^m  \vert L_{m-p}(m,n) \vert P_A ^{p})$ with some large number $N$, see Propositions \ref{mainprop1}, \ref{prop:5.6} and
\eqref{eq:1211dd}.
On the other hand, as in \eqref{eq:1119c}, together with {\tt (AF)}, we can see that
\begin{equation}\label{eq:1221a}
\vert L_{m-p}(m,n)\vert = \sum_{i,j \in \mathbb N, \; j \leq I(p+1,n)} A^{m-p-1}(i, j) \le \Vert L_A ^{m-p-1}\alpha ^{(I(p+1,n))} \Vert _{l^1}
\le \Vert L_A ^{m-p-1} \Vert _{l^1}
\end{equation}
where $\Vert L_A ^{m-p-1} \Vert _{l^1}$ is the operator norm of $L_A ^{m-p-1}$ and $\alpha ^{(k)} = \{ \alpha ^{(k)} _j\}_{j=1}^\infty \in l^1(D)$ is given by
$\alpha ^{(k)} _j =1$ for $j\le k$
and $=0$ for $j> k$.
Hence, the claim follows from the Gelfant formula $\lim_{m\to\infty}\Vert L_A^m\Vert ^{\frac{1}{m}} _{l^1}=r(L_A: l^1(D)\to l^1(D))$, by repeating the argument in the proof of the second inequality of Theorem \ref{mainthm1}. 
}
\[
ht(\gamma _A) \le \max\{ \log r (L_A:l^1 (D ) \to l^1(D )) ,  \log P_A\}. 
\]
Therefore, the following question is natural and important:
\begin{problem}
Under {\tt (SH)} and {\tt (UCF)}, does it hold that
\[
\log r (L_A:l^1 (D ) \to l^1(D ))   \le h_G(\sigma _A) + \mathcal I.
\]
\end{problem}

\subsection{Markov extension}
An important class of countable Markov shifts with an infinite alphabet appears in Hofbauer's Markov extension \cite{Hofbauer1986} (and its generalization by Buzzi \cite{Buzzi1997}) of subshifts with a finite alphabet, such as the $\beta$-shift with a non-algebraic $\beta >1$, the even shift and the Dyck shift.
See Appendix \ref{app:B1} for its definition.
We simply recall that $h_G(\sigma _{A(\Sigma ')}) = h_{\mathrm{top}}(\sigma _{\Sigma '})$, where $A(\Sigma ')$ is the transition matrix of a Markov extension of a subshift $\Sigma '$.
For a large class of Markov extensions $\Sigma _A$ (with $A=A(\Sigma ')$) of subshifts including the $\beta$-shift, it may hold that the \emph{transpose} matrix $A^{tr}$ of $A$ satisfies {\tt (SH)}, {\tt (FS)},  {\tt (AF)}, {\tt (SI)}, {\tt (SD)} and {\tt (O)} (refer to Appendix \ref{app:B1}), while it is known that $h_{\mathrm{top}} (\sigma _A) =h_{\mathrm{top}} (\sigma _{A^{tr}})$ when $A$ is a finite matrix.
However, it is unclear that one can even introduce a canonical cp map $\gamma _A$ on $\mathcal O_A$ (being canonical in the sense of Section \ref{subsec31}, especially satisfying the commutative diagram with $(\sigma _A)_*$) in such a case.
 Therefore, we leave the following problem.
 \begin{problem}
Assume that $A^{tr}$ satisfies {\tt (SH)}, {\tt (FS)}, {\tt (AF)},  {\tt (SI)}, {\tt (SD)} and {\tt (O)}. 
Then, does there exist a canonical cp map $\gamma _A:\mathcal O_A\to \mathcal O_A$ in the sense of Section \ref{subsec31}?
Moreover, $ht (\gamma _A) =ht (\gamma _{A^{tr}})$?
In particular, does it hold that $ht (\gamma _A) =h_{\mathrm{top}} (\sigma _{A})$, where $A$ is the Markov extension of the $\beta$-shift with a non-algebraic $\beta >1$?
 \end{problem}

In \cite{M2005}, Matsumoto considered the non-commutative entropy of the canonical cp map $\gamma _{\mathcal L(\Sigma ')}$ of his C$^*$-algebra $\mathcal O_{\mathcal L(\Sigma ')}$ associated with (the $\lambda$-graph $\mathcal L(\Sigma ')$ of) subshifts $\Sigma '$ over a finite symbol set. 
He showed that 
\[
ht(\gamma _{\mathcal L(\Sigma ')} \vert _{\mathcal D_{\mathcal L(\Sigma ')}}) =h_{\mathrm{top}}(\tilde{\sigma} _{\Sigma '}) >h_{\mathrm{top}}(\sigma _{\Sigma '})
\]
when $\Sigma '$ is the Dyck shift, where $\mathcal D_{\mathcal L(\Sigma ')}$ is the `diagonal' commutative C$^*$-subalgebra of $\mathcal O_{\mathcal L(\Sigma ')}$ and $\tilde{\sigma} _{\Sigma '}$ is the continuous left-shift operation on a compact Hausdorff space consisting of infinite words.
On the other hand, if $A$ satisfies {\tt (SH)}, {\tt (UCF)} and {\tt (FS)} and $X_A$ is compact, then 
\[
ht(\gamma _A\vert _{\mathcal D_A}) =h_{\mathrm{top}}(\hat \sigma _A) \ge h_G(\sigma _A).
\]
See Appendix \ref{app:B2} for details. 
The Markov extension of the Dyck shift satisfies these conditions (while {\tt (PCF)} is not satisfied; cf.~\cite{TY2022}), but it is unclear whether $h_{\mathrm{top}}(\hat \sigma _A) = h_G(\sigma _A)$.
Therefore, the following question seems to be natural.
 \begin{problem}
Assume that $A$ satisfies {\tt (SH)}, {\tt (UCF)} and {\tt (FS)}, and $X_A$ is compact. Then, does it hold that $h_{\mathrm{top}}(\hat \sigma _A) = h_G(\sigma _A)$? 
In particular,  does it hold for the Markov extension of the Dyck shift?
 \end{problem}

\section{Proof of Theorem \ref{thm:1}}

 Throughout this section, we assume that {\tt (SH)}, {\tt (FS)},  {\tt (AF)} and {\tt (PCF)}  hold for $A$ and there exists a KMS state on $\mO _A$. 
We first prepare notations for the proof of Theorem \ref{thm:1} (Section \ref{ss:1np}).
Next, we give the proof of Theorem \ref{thm:1} with $\left(\mO _A\right)^\Gamma$ instead of $\mO_A$ (Section \ref{ss:1goa}). In this step, the existence of a KMS state is not needed.
Finally, we complete the proof of Theorem \ref{thm:1} in Section \ref{ss:1oa}.
\subsection{Preliminary}\label{ss:1np}
Recall that $\vert \alpha \vert =m$ for  $\alpha = \{ \alpha _j\}_{j=1}^m \in  \Sigma _A^*\setminus \{e\}$ and $\vert e\vert =0$, where $e \in \Sigma _A^*$ is the empty word.
Recall also that $t(\alpha) = \alpha_m$ for $\alpha = \{ \alpha _j\}_{j=1}^m \in \Sigma _A^*\setminus \{e\}$ and $t(e)=1$.
Note that 
\begin{equation}\label{eq:1117}
A(i,j) =0 \quad \text{if and only if} \quad S_i S_j =0,
\end{equation}
so that $\alpha \not\in \Sigma _A^*$ if and only if $S_\alpha = S_{\alpha _1}\cdots S_{\alpha _m}=0$.
Let $w(m, n) \coloneqq  |L(m, n)|$ for $m,n\in \mathbb N$. Then, it follows from {\tt (UCF)} that $w(m, n)  <\infty$.

The following estimate is important to prove Theorem \ref{thm:1}. 
\begin{prop}\label{lem:la1109}
Assume that {\tt (SH)}, {\tt (UCF)} and {\tt (FS)}  hold.  
Then, for any sufficiently large $n\in \N$,
\[
h_G(\sigma _A) \le \liminf_{m\to\infty} \frac{1}{m} \log w(m,n) \le \limsup_{m\to\infty} \frac{1}{m} \log w(m,n) \le h_G(\sigma _A) +  \mathcal I. 
\]
\end{prop}
\begin{proof}
We first observe that, for each $m\ge 2$ and $n\in \N$,
\begin{equation}\label{eq:1119c}
\begin{split}
w(m, n) 
&= \left|\left\{ \{ \alpha _j\} _{j=1}^m  \in \N^m: \alpha_m \leq n, \prod_{k=1}^{m-1}A(\alpha_k, \alpha_{k+1}) =1 \right\}\right| \\
& = \sum_{i_1, \cdots, i_m \in \N, i_m \leq n} \prod_{k=1}^{m-1}A(i_k, i_{k+1}) = \sum_{i,j \in \mathbb N, \; j \leq n} A^{m-1}(i, j).
\end{split}
\end{equation}
Hence,  due to the fact that $h_G(\sigma _A)= \lim _{m\to \infty} \frac{1}{m}\log A^{m}(j,j)$ for any $j\in \N$ (cf.~\cite{Sariglecturenote}), we get $h_G(\sigma _A) \le \liminf_{m\to\infty} \frac{1}{m} \log w(m,n) $.
Below we will see the other inequality.

Notice also that $A$ is not locally finite due to {\tt (FS)} and the fact that $\mathcal O_A$ is unital (refer to Section \ref{ss:EL}).
So, there exists an $A$-infinite emitter, denoted by $J\in \N$ (otherwise, $A$ is row finite by {\tt (FS)}, implying that $A$ is locally finite due to {\tt (UCF)}).
Let 
\[
N_J\coloneqq \min \{ j_0\in \N : \forall j\ge j_0, \;  A(J,j) =1\}, 
\]
which is well-defined (and finite) because $J$ is an $A$-infinite emitter.

Fix $n\ge N_J$.
For $1\le i,j \le n$, we introduce
\[
k_j\coloneqq \min \{ l \in \N : A^l(j, J)\neq 0\}, \quad k_{i,j}\coloneqq \min \{ l \in \N : A^l(i,j)\neq 0\}.
\]
By the mixing property in {\tt (SH)}, 
\[
K\coloneqq  \max \left\{ \max _{ 1\le j \le n} (k_j +1), \max  _{ 1\le i,j \le n}k_{i,j}\right\} \le I(m,n).
\]

Let $j \in \{1,\ldots , n\}$.
When $i\ge n$, 
\begin{align*}
A^{m+ k_j+1}(j,j) &\ge A^{k_j}(j,J) A^{m+1}(J,j)\\
& \ge A^{m+1}(J,j)\\
& \ge A(J,i)A^m(i,j) = A^m(i,j).
\end{align*}
On the other hand, when $i <n$,
\begin{align*}
A^{m+ k_{i,j}}(j,j) \ge A^{k_{i,j}}(j,i) A^{m}(i,j) \ge A^{m}(i,j).
\end{align*}
Therefore, we have
 \[
\sum_{i,j\in \mathbb N, \; j \le n} A^{m-1}(i, j) \le n\cdot I(m,n)  \cdot \max _{0\le k\le K, \; 1\le j\le n} A^{m+k}(j,j).
 \]
 By using again the fact that $h_G(\sigma _A)= \lim _{m\to \infty} \frac{1}{m}\log A^{m}(j,j)$ for any $j\in \N$, we get
 \begin{multline*}
 \frac{1}{m} \log w(m,n) 
  \le  \frac{1}{m}\left(\log n +\log I(m,n) \right) + \max _{0\le k\le K, \; 1\le j\le n}  \frac{1}{m} \log A^{m+k}(j,j)\\
  \to h_G(\sigma _A) +\lim_{m\to\infty} \frac{1}{m}  \log I(m,n)
 \end{multline*}
 as $m\to \infty$.
 This completes the proof.
\end{proof}

\begin{remark}\label{rm:1227b}
For a locally finite transition matrix $A$,  the block entropy $h_b(A)$ and the loop entropy $h_l(A)$ can be written as
\begin{align*}
&h_b(A)= \limsup_{m\to\infty} \frac{1}{m} \log\left \vert \{ \alpha =\{ \alpha _k\} _{k=1}^m \in \Sigma _A ^*:  \alpha _1 =  j\}\right\vert ,\\
&h_l(A)= \limsup_{m\to\infty} \frac{1}{m} \log \left\vert \{ \alpha =\{ \alpha _k\} _{k=1}^m \in \Sigma _A ^*:  \alpha _1 = \alpha _m = j\}\right\vert 
\end{align*}
with some $j\in \mathbb N$. These quantities are independent of the choice of $j$ (cf.~\cite{JP06}).
As in the proof of Proposition \ref{lem:la1109}, one can see that
\[
\left\vert \{ \alpha =\{ \alpha _k\} _{k=1}^m \in \Sigma _{A^{tr}} ^*:  \alpha _1 =  j\}\right\vert = \sum_{i\in \mathbb N} A^m(i,j)
\]
and
\[
h_G(\sigma _A)  \le h_b(A^{tr}) \le h_G(\sigma _A) + \mathcal I.
\]
On the other hand, by Example \ref{ex:1227}, there exists a locally finite transition matrix $A$ with $\mathcal I >0$.
\end{remark}

\subsection{Non-commutative topological entropy for $(\mO_A)^\Gamma$}\label{ss:1goa}
We consider the canonical faithful $*$-representation $\pi : \mO_A\to B(l^2(\Sigma _A))$ of $\mO _A$, as before, and simply write $S_j$ for $\pi (S_j)$.
Similarly, we simply write $ \mathcal O_A \subset B(l^2(\Sigma _A))$ if it makes no confusion.

We will compute the non-commutative entropy of $\gamma _A\vert _{(\mathcal O_A)^\Gamma}$ by using the arguments in \cite{BG2000} for Markov shifts with finite symbols.
However, since we are considering Markov shifts with infinite symbols, for each $n, m \in \N$ the sum $\sum_{\alpha \in L(m, n)}S_\alpha S^* _\alpha$ is not equal to the unit of $\mO_A$.
In order to overcome the problem, we assumed {\tt (PCF)}.

Fix a minimal $A$-orthogonal division $H^{(1)}, \cdots, H^{(m_A)}$ with 
\[
P_A(H^{(1)}, \cdots, H^{(m_A)}) = P_A.
\]
Choose $n_1 \in \N$ and a finite subset $F$ of $\{1, \cdots, m_A\}$ with $|F|=P_A$ such that  we have 
\[
\N \backslash \{1, \cdots, n\} \subset \bigcup_{l \in F}H^{(l)}.
\]
for any $n \geq n_1$.
After an appropriate renumbering, we may assume that we have
\[
\N \backslash \{1, \cdots, n\} \subset \bigcup_{l =1}^{P_A} H^{(l)}.
\] 
for $n \geq n_1$.
Then, it follows from Lemma \ref{seqofsets} that for each $k \in \N$, $n \geq \max\{n_0, n_1\}$ and $l_k \in \{1, \cdots, P_A\}$,
\begin{equation}\label{eq:1119}
\N \setminus N(k,n) \subset \bigcup_{l _k=1}^{P_A} H^{(l_k)}.
\end{equation}
Using the strong operator topology on $B(l^2(\Sigma _A))$ (with which $\mO _A$ was identified), for each $k \in \N$, $n \geq \max\{n_0, n_1\}$ and $l_k \in \mathcal P_A\coloneqq \{1, \cdots, P_A\}$ we define
\[
S_{(k, n; l_k)}  \coloneqq  \sum_{\alpha \in H^{(l_k)}\setminus N(k, n)} S_\alpha .
\]
Then, by \eqref{eq:1119}, repeating the argument in Remark \ref{rem:1117}, we have 
\[
S_{(k,n;l_k)}S^* _{(k,n;l_k)} =\sum _{\alpha \in H^{(l_k)}\setminus N(k,n)} S_\alpha S_\alpha ^*.
\]
Thus, it follows from \eqref{eq:1111d} that
\begin{equation}\label{eq:1111c}
\sum_{\alpha \in N(k, n)} S_\alpha S^* _\alpha + \sum_{l_k \in \{1, \cdots, P_A\}} S_{(k,n;l_k)}S^* _{(k,n;l_k)} = 1.
\end{equation}
Furthermore, it follows from \eqref{eq:1117} and Lemma \ref{seqofsets} that for any $p' \ge p+1$ and $l_{p'}\in\mathcal P_A$, it holds that
\begin{equation}\label{eq:1117b}
S_{(p', n; l_{p'})}S_\alpha =0  \quad \text{ for all $\alpha \in N(p,n)$}.
\end{equation}
On the other hand, for any $p'\le p-1$ and $l_{p'}\in\mathcal P_A$, it does not necessarily hold that 
\[
S_\alpha S_{(p', n; l_{p'})} =0 \quad \text{ for all $\alpha \in N(p,n)$ }
\]
(e.g.~the renewal shift with $\alpha =1$).
Taking it into account,  we introduce the notation
\[
S_{\alpha (m-p,n; l_{m-p}) \cdots (1,n; l_{1})  } \coloneqq  S_\alpha S_{(m-p,n; l_{m-p})} \cdots S_{(1,n; l_{1})}
\]
for $\alpha \in L_p(m,n)$ and $(l_{m-p},\ldots ,l_{1})\in \mathcal P_A^{m-p}$ with $1\le p\le m-1$, and define 
\begin{equation*}
\overline{L(m, n)} \coloneqq  \bigcup_{p =0}^m\left\{ \alpha (m-p,n;l_{m-p}) \cdots (1,n;l_{1})  \, :\, \alpha \in L_p(m,n),\;   (l_{m-p},\ldots ,l_{1}) \in \mathcal P_A ^{m-p} \right\},
\end{equation*}
where the set inside the union over $p$ is interpreted as $L(m,n)$ when $p=m$ and as $\left\{ (m,n;l_m) \cdots (1,n;l_{1}) \, :\,  (l_m,\ldots ,l_{1}) \in \mathcal P_A^{m} \right\}$ when $p=0$.
Note that  $\overline{L(m, n)}$ is a finite set.
We write 
\[
\overline{w(m, n)} : = \left|\overline{L(m, n)}\right|.
\] 

\begin{prop}\label{prop:1117}
 Assume that {\tt (SH)}, {\tt (AF)} and {\tt (PCF)}   hold for $A$.
Then we have
\[
\sum_{\alpha \in \overline{L(m, n)}} S_\alpha S^* _\alpha =1
\]    
for any $m \in \N$ and $n\ge \max\{ n_0, n_1\}$. 
\end{prop}

\begin{proof}
Fix an $n\ge \max\{ n_0, n_1\}$. 
We prove the claim by induction for $m \in \N$.
For $m=1$, since $L(1,n)=N(1,n)$,  we already have
\[
\sum_{\alpha \in \overline{L(1, n)}}S_\alpha S^* _\alpha = \sum_{\alpha \in L(1, n)}S_\alpha S^* _\alpha + \sum_{l_1\in \mathcal P_A} S_{(1,n;l_1)}S^* _{(1,n;l_1)} =1
\]
by \eqref{eq:1111c}.
Assume that 
\begin{equation}\label{eq:1220}
\sum_{\alpha \in \overline{L(m, n)}} S_\alpha S^* _\alpha =1
\end{equation}
for some $m \in \N$.
We divide $\sum_{\alpha \in \overline{L(m+1, n)}} S_\alpha S^* _\alpha $ as $\sum_{\alpha \in \overline{L(m+1, n)}} S_\alpha S^* _\alpha = I+ II$, where
\begin{align*}
&I\coloneqq
 \sum _{p=0}^m \sum_{\substack{(l_{m-p},\ldots ,l_{1})\in \mathcal P_A^{m-p}\\ \beta \in L_{p+1}(m+1,n)}}
 S_\beta S_{(m-p,n;l_{m-p})}\cdots S_{(1,n;l_{1})}   S_{(1,n;l_{1})}^* \cdots  S_{(m-p,n;l_{m-p})}^*S^* _\beta ,\\
&II\coloneqq\sum_{(l_{m+1},\ldots ,l_{1})\in \mathcal P_A^{m+1}} S_{(m+1,n;l_{m+1})}\cdots S_{(1,n;l_{1})}   S_{(1,n;l_{1})}^* \cdots  S_{(m+1,n;l_{m+1})}^* .
 \end{align*}
 Notice that, for each $p\in \{0,\ldots , m\}$ and $f \in \mathcal O_A$,
 \[
\sum _{\beta \in L_{p+1}(m+1,n)} S_\beta f S_\beta^* = \sum _{\gamma \in N(m+1,n)} \sum _{\beta \in L_p(m,n)} S_\gamma S_\beta f S_\beta ^* S_\gamma^*
\]
due to \eqref{eq:1117}.
Thus,
\[
I=\sum_{\gamma \in N(m+1, n)} S_\gamma\left(\sum_{\alpha \in \overline{L(m, n)}} S_\alpha S^* _\alpha \right)S^* _\gamma = \sum_{\gamma \in N(m+1, n)} S_\gamma S^* _\gamma 
\]
by \eqref{eq:1220}.
On the other hand, it follows from  \eqref{eq:1117b} and \eqref{eq:1220} that
 \[
 II = \sum_{l_{m+1} \in \mathcal P_A } S_{(m+1,n;l_{m+1})}  \left(\sum_{\alpha \in \overline{L(m, n)}} S_\alpha S^* _\alpha\right) S_{(m+1,n;l_{m+1})}^* = \sum_{l_{m+1} \in \mathcal P_A } S_{(m+1,n;l_{m+1})}  S_{(m+1,n;l_{m+1})}^*.
 \]
Combining these observations, we get
\[
\sum_{\alpha \in \overline{L(m+1, n)}} S_\alpha S^* _\alpha
=\sum_{\gamma \in N(m+1, n)} S_\gamma S^* _\gamma +
\sum_{l_{m+1} \in \{1, \ldots ,P_A\} } S_{(m+1,n;l_{m+1})} S_{(m+1,n;l_{m+1})}^* =1
\]
by \eqref{eq:1111c}.
This completes the proof.
\end{proof}

Let $\left(\mO_A\right)^{''} \subset B(l^2(\Sigma _A))$ be the double commutant of $\mO_A$.
For each $m \in \N$ and $n \ge \max\{n_0,n_1\}$, we define a linear map $\rho_{m, n} \colon \mO_A \to M_{\overline{L(m, n)}} (\C )\otimes \left(\mO_A\right)^{''}$ via
\[
\rho_{m, n}(f) \coloneqq  \sum_{\mu, \nu \in \overline{L(m, n)}} e_{\mu, \nu} \otimes S^* _\mu f S_\nu
\]
for $f \in \mO_A$. 
Here, $M_{\overline{ L(m,n)}}(\C )$ is the C$^*$-algebra consisting of matrices $E=(E(\alpha , \beta))_{\alpha , \beta \in \overline{ L(m,n)}}$ (abbreviated as $M_{\overline{ L(m,n)}}$ when it makes no confusion), 
and $e_{\mu, \nu} = (e_{\mu, \nu}(\alpha , \beta))_{\alpha , \beta \in \overline{ L(m,n)}}$ is the $0$-$1$ matrix such that $e_{\mu, \nu}(\alpha , \beta)=1$ if and only if $(\alpha , \beta) =(\mu, \nu)$.
Since we have $\sum_{\alpha \in \overline{L(m, n)}} S_\alpha S^* _\alpha=1$ by Proposition \ref{prop:1117}, the map $\rho_{m, n}$ is a $*$-homomorphism.
In addition, the simplicity of $\mO_A$ implies that $\rho_{m, n}$ is injective.

\begin{prop}\label{mainprop1}
Assume that {\tt (SH)}, {\tt (FS)}, {\tt (AF)} and {\tt (PCF)} hold for $A$.
Then one has 
\[
\ga_A|_{\left(\mO_A\right)^\Ga} \leq \sup_{n \ge \max \{n_0,n_1\}}\left(\limsup_{m\to \infty} \frac{1}{m}\log \overline{w(m, n)}\right).
\]
\end{prop}

\begin{proof}
For each natural number $m_0, n$, we write 
\[
\omega_A(m_0, n) \coloneqq  \left\{ S_\alpha S_\beta ^* : \text{$\alpha , \beta \in L(\bar m, n)$ with $ \bar m\in\{1,\ldots , m_0\}$} \right\} \cup \{ 1\} \subset \omega _A
\]
(recall \eqref{eq:1117e} for $\omega _A$).
By \eqref{eq:1221b} and the property in Section \ref{subsec23}, it is sufficient to show that for all $\delta >0, m_0 \in \N$ and $n \geq \max\{n_0,n_1\}$ one has
\[
\limsup_{m \to \infty} \frac{1}{m} \log rcp\left(\bigcup_{k=0}^{m-1} \ga_A ^{k}\left(\omega_A(m_0, n)\right); \delta\right) \leq \limsup_{m\to \infty} \frac{1}{m}\log \overline{w(m, n)}.
\]
Let $\mathcal F_n$ be the C$^*$-subalgebra of $\left(\mO_A\right)^{''}$ generated by $\{ S^* _{\kappa} S_{\kappa} : \kappa \in  \overline{L(1, n)}\}$.
Since $\mathcal F_n$ is commutative, in particular exact, and $\{ S^* _{\kappa} S_{\kappa} : \kappa \in  \overline{L(1, n)}\}$ is a finite set, there are ccp maps 
\[
\phi_1 \colon \mathcal F_n \to M_{m_1}(\C), \quad \psi_1 \colon M_{m_1}(\C) \to B(l^2(\Sigma _A))
\]
for some $m_1\in\mathbb{N}$ with
\[
\| \psi_1\phi_1\left(S_\kappa ^* S_\kappa\right) - S_\kappa ^* S_\kappa \| \leq \delta \left( \overline{w(1, n)}\right)^{-1}
\]
for any $\kappa \in \overline{L(1, n)}$.
Hence, using Arveson's extension theorem,\footnote{Arveson's extension theorem states that for a C$^*$-subalgebra $\mathcal F$ of a C$^*$-algebra $\mathcal O$, any cp map $\phi $ from $\mathcal F$ to $B(\mathcal H)$ with a Hilbert space $\mathcal H$ can be extended to a cp map from $\mathcal O$ to $B(\mathcal H)$ (cf.~\cite{NS2006}).}
one can extend $\phi_1$ to a ccp map from $(\mathcal O_A)^{''}$ to $M_{m_1}(\C )$, which is also denoted by $\phi _1$. 

For each $m \in \N$, let 
\[
\Psi_{m, n}\colon M_{\overline{L(m, n)}}(\C ) \otimes B(l^2(\Sigma _A)) \to B(l^2(\Sigma _A))
\]
be a ccp extension of $\rho_{m, n} ^{-1}$.
Write $m'\coloneqq  m+m_0$ and set 
\[
\phi\coloneqq  (id_{M_{{\overline{L(m', n)}}}(\C )} \otimes \phi_1)\rho_{m', n} \ , \ \psi \coloneqq \Psi_{m', n}(id_{M_{{\overline{L(m', n)}}}(\C )} \otimes \psi_1).
\]
See the following diagram.
\vspace{0.1cm}
$$\xymatrix@M=4pt{
\mbox{\footnotesize $\mathcal O_A$} \ar@{->} [rr] ^(0.45){\rho_{m',n}}  
\ar@{->} [ddrrr] _(0.4){\phi} & & \ar @{} [d] |{} 
\mbox{\footnotesize $\rho_{m',n}(\mathcal O_A)$}
\ar@{->} [ddr] ^{id \otimes \phi_1} \ar@{->} [rr] ^{id_{\rho_{m',n}(\mathcal O_A)}}
 &  & 
\ar @{} [dr] < 8pt> |{} \mbox{\footnotesize $\rho_{m',n}(\mathcal O_A)$}
 \ar@{->} [rr] ^{\rho_{m',n}^{-1}} 
\ar@{^{(}->} [d] &  & \mbox{\footnotesize $B (l^2(\Sigma _A))$}  \\ 
&   &  &   & \mbox{\footnotesize $M_{\overline{L(m',n)}}  \otimes \mathcal O_A$}
\ar @{} [d]  |{}  \ar@{->} [urr] _{\Psi_{m',n}} &  & \\
         &     &   & 
\mbox{\footnotesize $M_{\overline{L(m',n)}} \otimes M_{m_1} $}  
\ar@{->} [ur] _(0.55){id \otimes \psi_1} 
\ar@/_3pc/ @{->} [uurrr] _(0.65){\psi} &   &   &  }$$

\vspace{0.1cm}

The argument in \cite[Lemma 2]{BG2000}, with a modification due to the introduction of $\overline{L(m,n)}$, shows the following.
\begin{lem}\label{lem:1113}
Assume that {\tt (AF)} holds.
Then, for any $f \in \omega_A(m_0, n), m\in \N$ and $l=0, \cdots, m-1$, there exists a set of partial isometries $\{ X_{\kappa} : {\kappa} \in \overline{L(1, n)} \} \subset M_{{\overline{L(m', n)}}}$ with
\[
\rho_{m+m_0, n}\ga_A ^l(f) = \sum_{{\kappa} \in \overline{L(1, n)}} X_{\kappa} \otimes S^* _{\kappa} S_{\kappa}.
\]
\end{lem}
\begin{proof}
Let $m'\coloneqq  m+m_0$. 
We first consider the case $f= S_\alpha S_\beta ^*$ with $\alpha = \{\alpha _i\}_{i=1}^{\bar m}, \beta =\{\beta _i\}_{i=1}^{\bar m} \in L(\bar m,n)$ for some $\bar m\in \{1,\ldots , m_0\}$.
Then, it follows from \eqref{eq:1117}  that
\begin{align*}
\rho_{m',n} \gamma_A^l (S_\alpha S_\beta^*) 
&=\sum_{\mu,\nu \in \overline{L(m',n)}}\sum_{\eta \in \Sigma _A^*,\; \vert \eta \vert =l} 
e_{\mu , \nu} \otimes S_\mu^* S_\eta S_\alpha  S_\beta^* S_\eta^* S_\nu \\
&=\sum_{\mu,\nu \in \overline{L(m',n)}}\sum_{\eta \in L_{l}(l+\bar m,n)} 
e_{\mu , \nu} \otimes S_\mu^* S_\eta S_\alpha  S_\beta^* S_\eta^* S_\nu .
\end{align*}
For $\mu =\{\mu _i\}_{i=1}^{m'}$ and $\nu =\{\nu _i\}_{i=1}^{m'}$, write 
\[
\mu '\coloneqq\{\mu _i\}_{i=l+1}^{m'}, \; \mu ''\coloneqq\{\mu _i\}_{i=l+\bar m+ 1}^{m'},  
\; \nu '\coloneqq\{\nu _i\}_{i=l+1}^{m'}, \; \nu ''\coloneqq\{\nu _i\}_{i=l+\bar m + 1}^{m'}.
\]
Then, $\vert \mu ''\vert =\vert \nu '' \vert = m' - l -m_0 \ge 1$ and $t(\mu'')=\mu_{m'}$.
For $\eta =\{\eta _i\}_{i=1}^{l}$, it follows from \eqref{EL2}, \eqref{EL3} and Lemma \ref{seqofsets} that
\begin{align*}
S_\mu^* S_\eta & =(S_{\mu _2}\cdots S_{\mu_{m'}})^* (S_{\mu_1}^* S_{\eta _1}S_{\eta _2}) (S_{\eta _3}\cdots S_{\eta _l}) \\
&=
\begin{cases}
(S_{\mu _2}\cdots S_{\mu_{m'}})^* (S_{\eta _2}\cdots S_{\eta _l}) \quad &(\text{if $\mu _1=\eta _1$})\\
0 \quad &(\text{otherwise})
\end{cases}.
\end{align*}
More precisely, if $\mu _1$ belongs to $\N$, then this is just a consequence of \eqref{EL2} and \eqref{EL3}.
Otherwise (so that $\mu _1\neq \eta _1$), by the construction of $\overline{L(m',n)}$,  $\mu _{1} =(m',n; l_{m'})$ with some $l_{m'}\in \mathcal P_A$. 
Thus, observing that $\eta _1\in N(l+\bar m,n) \subset N(m',n)$ by Lemma \ref{seqofsets} (recall that $\eta \in L_l(l+\bar m,n)$), we have
\[
S_{\mu_1}^* S_{\eta _1}=\sum _{\gamma \in H^{(l_{m'})} \setminus N(m',n)} S_{\gamma }^* S_{\eta _1} = 0
\]
by \eqref{EL2}.

Similarly, due to  \eqref{EL2}, \eqref{EL3} and  Lemma \ref{seqofsets}, we get
\[
S_\mu^* S_\eta S_\alpha  S_\beta^* S_\eta^* S_\nu = 
\begin{cases}
S_{ \mu'}^* (S_{\eta_l}^*S_{\eta _l})S_\alpha  S_\beta^* (S_{\eta _l}^* S_{\eta _l} ) S_{\nu '} \quad &(\text{if $\mu = \eta \mu '$ and $\nu = \eta \nu '$}) \\
0 \quad &(\text{otherwise})
\end{cases}.
\]
 Indeed, for the most involved part, observe that if   $\mu _j =(m' -j +1,n; l_{m'-j+1})$ for some $l_{m'-j+1}\in \mathcal P_A$ with $1\le j\le l$, then 
  \[
  S_{\mu_j}^* S_{\eta _j}=\sum _{\gamma \in H^{(l_{m'-j+1})} \setminus N(m'-j+1,n)} S_{\gamma }^* S_{\eta _j} = 0
\]
because $\eta _j\in N(l+\bar m -j +1,n) \subset N(m'-j+1,n)$.
 It also follows  from \eqref{EL3} that
\[
S_{ \mu'}^* (S_{\eta_l}^*S_{\eta _l})S_\alpha  S_\beta^* (S_{\eta _l}^* S_{\eta _l} ) S_{\nu '} 
=
\begin{cases}
S_{ \mu'}^* S_\alpha  S_\beta^* S_{\nu '} \quad &(\text{if  $\eta \alpha , \eta \beta \in \Sigma _A^*$})\\
0 \quad & (\text{otherwise})
\end{cases}.
\]

Furthermore, by using \eqref{EL2}, \eqref{EL3} and Lemma \ref{seqofsets} again, we get
\[
S_{ \mu'}^* S_\alpha  S_\beta^* S_{\nu '} = 
S_{ \mu''}^* (S_{\alpha _{\bar m}}^*S_{\alpha _{\bar m}} )( S_{\beta _{\bar m}}^*  S_{\beta _{\bar m}} )S_{\nu ''} = S_{\mu ''} ^*S_{\nu ''}  
\]
if $\mu ' = \alpha \mu ''$, $\nu '= \beta \nu ''$ and $\alpha \mu '' , \beta \nu '' \in \Sigma _A^*$, 
and $S_{ \mu'}^* S_\alpha  S_\beta^* S_{\nu '} = 0$ otherwise.
As before, observe that if   $\mu _{l+j} =(m' -l - j +1,n; l_{m' -l -j+1})$ for some $l_{m' -l -j+1}\in \mathcal P_A$ with $1\le j\le \bar m$, then 
  \[
  S_{\mu_{l+j}} ^*S_{\alpha _j}=\sum _{\gamma \in H^{(l_{m'-l-j+1})} \setminus N(m'-l-j+1,n)} S_{\gamma } ^*S_{\alpha _j} = 0
\]
because $\alpha _j\in N(\bar m -j +1,n) \subset N(m'-l-j+1,n)$.
Finally, note that for each $j \in\{l+\bar m+1,\ldots , m'\}$,  we have 
\[
S^*_{\mu _j} S_{\nu _j} = 
\begin{cases}
S^*_{\mu _j} S_{\mu _j} \quad & (\text{if $\mu _j= \nu _j$})\\
0 \quad & (\text{otherwise})
\end{cases}
\]
by \eqref{EL2} and the fact that $\N$ is decomposed into mutually disjoint subsets $N(j,n)$, $H^{(l_1)} \setminus N(j,n)$, $\ldots $, $H^{(l_{P_A})} \setminus N(j,n)$. 
In particular, observe that if $\mu _j=(m'-j+1,n;l_{m'-j+1})$ and $\nu _j=(m'-j+1,n;l_{m'-j+1}')$ for some $l_{m'-j+1}, l_{m'-j+1}' \in \mathcal P_A$ with $l_{m'-j+1} \neq l_{m'-j+1}'$, then
\[
S^*_{\mu _j} S_{\nu _j} = \sum _{\gamma \in H^{(l_{m'-j+1})} \setminus N(m'-j+1,n)} \, \sum _{\gamma ' \in H^{(l_{m'-j+1}')} \setminus N(m'-j+1,n)} S_{\gamma }^* S_{\gamma '} =0
\]
by \eqref{EL2}. 
Hence, we get
\[
S_{\mu ''} ^*S_{\nu ''} = 
\begin{cases}
S_{t(\mu '')}^*S_{t(\mu '')} \quad &(\text{if  $\mu ''=\nu ''$})\\
0 \quad &(\text{otherwise})
\end{cases}.
 \]

In conclusion, we get
\begin{align*}
\rho_{m',n} \gamma_A^l (S_\alpha S_\beta^*) 
= \sum_{\eta \in L_{l}(l+\bar m,n)} 
\sum_{\substack{\mu '' \in \overline{L(m'-l-\bar m,n)}\\ 
\eta \alpha \mu^{\prime \prime},\eta \beta \mu^{\prime \prime}
\in \Sigma _A^*}} 
e_{\eta \alpha \mu^{\prime \prime},\eta \beta \mu^{\prime \prime}}
\otimes S_{t(\mu '')}^*S_{t(\mu'')} =\sum_{\kappa \in \overline{L(1, n)}} X_{\kappa} \otimes S^* _{\kappa} S_{\kappa}
\end{align*}
with
\[
X_{\kappa} =\sum_{\eta \in L_{l}(l+\bar m,n)} 
\sum_{\substack{\mu '' \in \overline{L(m'-l-\bar m,n)}\\ 
\eta \alpha \mu^{\prime \prime},\eta \beta \mu^{\prime \prime}
\in \Sigma _A^*, \; t(\mu '') =\kappa }} e_{\eta \alpha \mu^{\prime \prime},\eta \beta \mu^{\prime \prime}} \in M_{\overline{L(m',n)}}(\C ).
\]
Each $X_{\kappa} $ is a partial isometry because for each $\gamma , \gamma ' \in \overline{L(m',n)}$,  the $\gamma$-row of the matrix $X_{\kappa}$ is a non-zero vector only if $\gamma =\eta \alpha \mu^{\prime \prime}$ with some $\eta \in L_{l}(l+\bar m,n)$, $\mu '' \in \overline{L(m'-l-\bar m,n)}$ satisfying $\eta \alpha \mu^{\prime \prime},\eta \beta \mu^{\prime \prime} \in \Sigma _A^*$, $t(\mu '') =\kappa $, 
and in the non-zero vector case, all the $\gamma '$-entry of the $\gamma$-row is $0$ except the $\eta \beta \mu^{\prime \prime}$-entry with value $1$.

For the case $f=1$,  a similar computation yields
 \begin{align*}
\rho_{m',n} \gamma_A^l (1) 
=\sum_{\kappa \in \overline{L(1, n)}} X_{\kappa} \otimes S^* _{\kappa} S_{\kappa}
\end{align*}
with
\[
X_{\kappa} =\sum_{\eta \in L_{l}(l+\bar m,n)} 
\sum_{\substack{\mu ' \in \overline{L(m'-l,n)}\\ 
\eta  \mu^{\prime }
\in \Sigma _A^*, \; t(\mu ') =\kappa }} e_{\eta  \mu^{\prime},\eta  \mu^{\prime}} \in M_{\overline{L(m',n)}}(\C ).
\]
This completes the proof.
\end{proof}
By Lemma \ref{lem:1113}, for any $f \in \omega_A(m_0, n), m\in \N$ and $l=0, \cdots, m-1$ one has
\[
\left\| (id_{M_{{\overline{L(m', n)}}}} \otimes \psi_1 \phi_1) \left(\rho_{m', n}\ga_A ^l(f)\right) - \rho_{m', n}\ga_A ^l(f) \right\| \leq \delta.
\]
Thus, we obtain
\begin{equation*}
\begin{split}
\left\| \psi \phi \left(\ga_A ^l(f) \right) - \ga_A ^l(f) \right\| 
& = \left\| \Psi_{m', n}(id_{M_{{\overline{L(m', n)}}}} \otimes \psi_1 \phi_1) \left(\rho_{m', n}\ga_A ^l(f)\right) - \Psi_{m', n}\rho_{m', n}\ga_A ^l(f) \right\| \\
&\leq \left\| (id_{M_{{\overline{L(m', n)}}}} \otimes \psi_1 \phi_1) \left(\rho_{m', n}\ga_A ^l(f)\right) - \rho_{m', n}\ga_A ^l(f) \right\| \leq \delta.
\end{split} 
\end{equation*}
This inequality implies that we have
\[
\frac{1}{m} \log rcp\left(\bigcup_{k=0}^{m-1} \ga_A ^{k}\left(\omega_A(m_0, n)\right); \delta\right) \leq \frac{1}{m}\log \left(m_1\overline{w(m+m_0, n)}\right).
\]
for any $m \in \N$.
Thus, we get
\[
\limsup_{m \to \infty} \frac{1}{m} \log rcp\left(\bigcup_{k=0}^{m-1} \ga_A ^{k}\left(\omega_A(m_0, n)\right); \delta\right) \leq \limsup_{m\to \infty} \frac{1}{m}\log \overline{w(m, n)}.
\]
This completes the proof of Proposition \ref{mainprop1}.
\end{proof}

\begin{thm}\label{mainthm1}
Assume that {\tt (SH)}, {\tt (FS)}, {\tt (AF)} and {\tt (PCF)} hold for $A$.
Then one has
\[
ht(\ga_A|_{\left(\mO_A\right)^\Ga} )\leq   h_G(\sigma _A)   + \log P_A  +\mathcal I
\]
and
\[
ht(\ga_A|_{\left(\mO_A\right)^\Ga} )\leq \max\{ h_G(\sigma _A),   \log P_A\} + \mathcal I+\mathcal D.
\]
\end{thm}
\begin{proof}
By Proposition \ref{mainprop1}, since the sequence $\displaystyle \left\{ \limsup_{m\to \infty} \frac{1}{m}\log \overline{w(m, n)} \right\}_{n \in \N}$ is increasing, it is sufficient to show that 
\begin{equation}\label{eq:1211b}
\limsup_{m\to \infty} \frac{1}{m}\log \overline{w(m, n)} \leq  h_G(\sigma _A)   + \log P_A  +\mathcal I
\end{equation}
and
\begin{equation}\label{eq:1211c}
\limsup_{m\to \infty} \frac{1}{m}\log \overline{w(m, n)} \leq \max\{ h_G(\sigma _A),   \log P_A\} +\mathcal I +\mathcal D
\end{equation}
for any large $n \in \N$.
By the construction of $\overline{L(m,n)}$, one has
\begin{align}\label{eq:1211dd}
\overline{w(m, n)} = \sum_{p=0}^m  \vert L_{m-p}(m,n) \vert P_A ^{p}
\end{align}
with $\vert L_0(m,n) \vert \coloneqq 1$.
Since $\vert L_{m-p}(m,n)\vert \le \vert L(m,n)\vert =w(m,n)$, it follows from Proposition \ref{lem:la1109} that
\[
 \frac{1}{m}\log \overline{w(m, n)} \le
  \frac{1}{m}\log w(m,n) +   \frac{1}{m}\log \frac{P_A^{m+1}-1}{P_A-1} \to h_G(\sigma _A) +\mathcal I + \log P_A
\]
as $m\to \infty$, so \eqref{eq:1211b} holds.

Next we prove \eqref{eq:1211c}.
By   the construction of $D(m,n)$, we have
\begin{align*}
 \vert L_{m-p}(m,n) \vert P_A ^{p} \le  D(m,n) \cdot  \frac{w(m,n)}{w(p,n)} \cdot P_A ^{p} 
\end{align*}
for each $p\in \{1, \ldots ,m-1\}$.
On the other hand, it follows from Proposition \ref{lem:la1109} that for any $\epsilon >0$, there is an integer $p_0$ such that 
 \[
 e^{p(h_G(\sigma _A) -\epsilon )}\le w(p,n) \le e^{p(h_G(\sigma _A) + \mathcal I +\epsilon )}
 \]
 for all $p \ge p_0$ and large $n$.
 On the other hand, if $p< p_0$ and $m$, $n$ are large, then
 \[
 \vert L_{m-p}(m,n) \vert P_A ^{p}\le \vert L(m,n)\vert P_A^{p_0}  \le e^{m(h_G(\sigma _A) + \mathcal I +\epsilon )} P_A^{p_0} 
 \]
 by  Proposition \ref{lem:la1109}.
 Thus,   we have
 \begin{align*}
 \overline{w(m, n)} & \le p_0 P_A^{p_0} e^{m(h_G(\sigma _A) + \mathcal I +\epsilon )} 
 +D(m,n)  \sum _{p=p_0}^m e^{m(h_G(\sigma _A) + \mathcal I +\epsilon ) - p(h_G(\sigma _A) -\epsilon )}P_A^p \\
& \le 
D(m,n) C_{\epsilon} \left(  e^{m(h_G(\sigma _A) + \mathcal I +\epsilon )} + e^{m (\mathcal I +2\epsilon )} P_A^m  + 1 \right)
 \end{align*}
 with some positive constant $C_{\epsilon}$  being independent of $m$ (notice that $D(m,n) \ge 1$ by construction).
 Therefore, since
 \[
\limsup _{m\to\infty} \frac{1}{m}\log \left(a _m +b_m\right) = \max\left\{ \limsup _{m\to\infty} \frac{1}{m}\log a _m , \limsup _{m\to\infty} \frac{1}{m}\log b_m\right\} 
 \]
 for any sequences $\{ a_m\} _{m\in \N}$, $\{ b_m\} _{m\in \N}$ of positive numbers, we get
\begin{equation*}
\limsup_{m\to \infty} \frac{1}{m}\log \overline{w(m, n)}  
\le \max\left\{h_G(\sigma _A) +\mathcal I +\epsilon ,
  \log P_A + \mathcal I +2\epsilon \right\} +\mathcal D.
\end{equation*}
Since $\epsilon$ is arbitrary, we complete the proof of \eqref{eq:1211c}.
\end{proof}

\subsection{Non-commutative topological entropy for $\mO_A$}\label{ss:1oa}
In this section, we additionally assume the existence of a KMS state for the gauge action $\Ga$ and compute the non-commutative entropy $ht(\ga_A)$. 
Let $\varphi$ be a KMS state for $\Ga$ with some inverse temperature satisfying that
\[
\sum_{i=1}^\infty \varphi\left(S_i S_i ^*\right) =1.
\]
We denote  the GNS representation of $\varphi$ by $(\pi_\varphi, \mathcal{H}_\varphi)$. 
We identify $\pi_\varphi(\mO_A)$ with $\mO_A$.
Let $\left(\mO_A\right)^{''}$ be the double commutant of $\mO _A$ in $B(\mH_\varphi)$. 
Then the canonical gauge action $\Ga$ extends to an action of $\T$ on $\left(\mO_A\right)^{''}$. 
For each $p \in \Z$, we define a linear map
\[
E_p\colon \left(\mO_A\right)^{''} \to \left(\mO_A\right)^{''}
\]
by
\[
E_p(f)\coloneqq \int_\T z^{-p}\Ga_z(f)\ dz, \quad f \in \left(\mO_A\right)^{''},
\]
where $dz$ is the Haar measure on $\mathbb T$.
Let $\overline{\left(\mO_A\right)^{\Ga}}$ be the closure of $\left(\mO_A\right)^{\Ga}$ with respect to the strong operator topology.
Note that $E_0$ is a conditional expectation onto $\overline{\left(\mO_A\right)^{\Ga}}$ (and thus it is a ccp map).

For $m, n, m_0 \in \N$, we set
\[
\Omega_A(m, n) \coloneqq  \left\{ S_\alpha S_\beta ^* \in \Omega_A\setminus \{1\} : |\alpha| \leq m, |\beta| \leq m, t(\alpha) \leq n, t(\beta) \leq n \right\} \cup \{ 1\},
\]
\[
\mathcal{L}(m, n, m_0)\coloneqq \bigcup_{k=m}^{m+m_0} \overline{L(k, n)} \quad \text{and} \quad
\mathcal{W}(m, n, m_0)\coloneqq \left| \mathcal{L}(m, n, m_0) \right|.
\]
Notice that for each $f \in \Omega_A(m,n)$ with $m,n\in \N$, there is $p\in \Z$ with $\vert p\vert \le m$ such that $\Gamma _z(f) =z^p f$ for each $z\in \mathbb T$ because $\Gamma _z(S_\alpha S_\beta ^*) = z^{\vert \alpha \vert - \vert \beta \vert } S_\alpha S_\beta ^*$ for $\alpha ,\beta \in \Sigma _A^*$, and that
\begin{equation}\label{eq:1221c}
E_q(f) =
\begin{cases}
f \quad &(q=p)\\
0 \quad &(\text{otherwise})
\end{cases}
\end{equation}
when $\Gamma _z(f) =z^p f$.

\begin{remark}\label{rem:1228b}
Let $\varphi$ be a KMS state for $\Ga$ with some inverse temperature such that
\[
\sum_{i=1}^\infty \varphi\left(S_i S_i ^*\right) =1.
\]
Then one has
\[
\sum_{j=1}^\infty S_j S_j ^* =1
\]
with respect to the strong operator topology on $\mathcal{B}(\mathcal{H}_\varphi)$.
Indeed, for any $i \in \N$, a computation shows
\begin{equation}\label{RemEq}
\left(1- \sum_{j=1}^\infty S_j S_j ^*\right)S_i =0 
\end{equation}
and
\begin{equation}\label{RemEqStar}
\begin{split}
\lim_{n \to \infty} \varphi\left(S_i \left(1- \sum_{j=1}^n S_j S_j ^*\right)^2 S_i ^*\right)
&=\lim_{n \to \infty} e^{-\beta}\varphi\left(\left(1- \sum_{j=1}^n S_j S_j ^*\right)^2 S_i ^*S_i \right) \\
&=\lim_{n \to \infty} e^{-\beta}\varphi\left(\left(1- \sum_{j=1}^n S_j S_j ^*\right) S_i ^*S_i \left(1- \sum_{j=1}^n S_j S_j ^*\right)\right) \\
&\leq \lim_{n \to \infty} e^{-\beta}\| S_i \|^2 \varphi\left(1- \sum_{j=1}^n S_j S_j ^*\right) =0.
\end{split}
\end{equation}
Since $\mathcal{O}_A$ is generated by $\{ S_i \}_{i \in \N}$, equations (\ref{RemEq}) and (\ref{RemEqStar}) imply the required convergence. 
\end{remark}
Using the strong operator topology on $B(\mathcal{H}_\varphi)$, we define 
\[
S_{(k, n; l_k)}  \coloneqq  \sum_{\alpha \in H^{(l_k)}\setminus N(k, n)} S_\alpha.
\]
for $k\in \N$, $n \in \mathbb N$ and $l_k \in \mathcal P_A$.
From now on we do not consider $l^2(\Sigma_A)$.
Thus, the symbol $S_{(k, n; l_k)}$ does not cause confusion.
Due to Remark \ref{rem:1228b}, one can show Proposition \ref{prop:1117} for these new $S_{(k, n; l_k)}$'s.

We consider the following ccp maps motivated by \cite{Matsumoto2002,PWY2000}.
The first map 
\[
\Phi_{m, n, m_0} \colon \mO_A \to M_{\mathcal{L}(m, n, m_0)} \otimes \overline{\left(\mO_A\right)^{\Ga}}
\]  
is given by
\[
\Phi_{m, n, m_0}(f) \coloneqq  \sum_{\mu, \nu \in \mathcal{L}(m, n, m_0)} e_{\mu, \nu} \otimes S^*_\mu E_{|\mu| -|\nu|}(f)S_\nu, \quad f \in \mO_A.
\]
\begin{lem}
Assume that {\tt (FS)} {\tt (AF)} and {\tt (PCF)} hold for $A$.
Assume also that there exists a KMS state $\varphi$ for the canonical gauge action $\Ga$ with
\[
\sum_{i=1}^\infty \varphi\left(S_i S_i ^*\right) = 1.
\]
Let $m, n, m_0 \in \N$. 
Consider 
\[
V\coloneqq  \mathcal{W}(m, n, m_0)^{-\frac{1}{2}} \sum_{\mu, \nu \in \mathcal{L}(m, n, m_0)} e_{\mu, \nu} \otimes S _\nu \in M_{\mathcal{L}(m, n, m_0)} \otimes \left(\mO_A\right)^{''}.
\]
Then one has
\[
\Phi_{m, n, m_0}(f) = (id_{M_{\mathcal{L}(m, n, m_0)}} \otimes E_0)(V^* (1 \otimes f)V)
\]
for any $f \in \mO_A$.
In particular,
$\Phi_{m, n, m_0}$ is a ccp map.
\end{lem}

\begin{proof}
Fix any $S_\alpha S^* _\beta \in \Omega_A\setminus \{1\}$.
We have
\[
S^* _\mu E_{|\mu| - |\nu|}(S_\alpha S^* _\beta)S_\nu = \delta_{|\mu|-|\nu|, |\alpha|-|\beta|}S^* _\mu S_\alpha S^* _\beta S_\nu = E_0 \left(S^* _\mu S_\alpha S^* _\beta S_\nu \right)
\]
and
\[
S^* _\mu E_{|\mu| - |\nu|}(1)S_\nu = \delta_{|\mu|, |\nu|}S^* _\mu S_\nu = E_0 \left(S^* _\mu 1 S_\nu \right)
\]
for any $\mu, \nu, \in \mathcal{L}(m, n, m_0)$.
Since $\mO_A$ is generated by $\Omega_A \cup \{ 1 \}$, we obtain 
\[
S^* _\mu E_{|\mu| - |\nu|}(f)S_\nu = E_0 \left(S^* _\mu f S_\nu \right)
\]
for any $f \in \mO_A$.
This shows
\begin{equation*}
\begin{split}
\Phi_{m, n, m_0}(f) 
&= \sum_{\mu, \nu \in \mathcal{L}(m, n, m_0)} e_{\mu, \nu} \otimes S^* _\mu E_{|\mu| - |\nu|}(f)S_\nu \\
&= \sum_{\mu, \nu \in \mathcal{L}(m, n, m_0)} e_{\mu, \nu} \otimes E_0 \left(S^* _\mu f S_\nu \right) \\
&= (id_{M_{\mathcal{L}(m, n, m_0)}} \otimes E_0) \left( \sum_{\mu, \nu \in \mathcal{L}(m, n, m_0)} e_{\mu, \nu} \otimes S^* _\mu f S_\nu \right)
\end{split}
\end{equation*}
On the other hand, one computes
\begin{equation*}
\begin{split}
V^* (1 \otimes f)V &= \mathcal{W}(m, n, m_0)^{-1}\left( \sum_{\mu, \nu \in \mathcal{L}(m, n, m_0)} e_{\mu, \nu} \otimes S^* _{\mu} \right) (1 \otimes f) 
\left(\sum_{\mu', \nu' \in \mathcal{L}(m, n, m_0)} e_{\mu', \nu'} \otimes S _{\nu'} \right) \\
&=\mathcal{W}(m, n, m_0)^{-1} \left( \sum_{\mu, \nu '\in \mathcal{L}(m, n, m_0)} e_{\mu, \nu'} \otimes \mathcal{W}(m, n, m_0) S^*  _{\mu} f S_{\nu'}\right) \\
&= \sum_{\mu, \nu \in \mathcal{L}(m, n, m_0)} e_{\mu, \nu} \otimes S^* _\mu f S_\nu .
\end{split}
\end{equation*}
Thus, we obtain the conclusion.
\end{proof}

For each $\delta >0$ and $m, m_0 \in \N$, we choose $\xi  \equiv \xi _{m,m_0,\delta }\in l^2(\Z)$ with $\|\xi\|_2 \leq 1$ and 
\[
\lv  \sum_{\substack{s, r\in\{m,\ldots , m+m_0\} \\  s- r =p}} \xi (s)\overline{\xi (r)} -1 \rv \leq \frac{\delta}{2}
\]
for any $-(m+m_0) \leq p \leq m+m_0$.
The existence of such a function $\xi$ is ensured by Proposition 7.3.8 of \cite{EO2018} due to the amenability of $\Z$ (cf.~\cite[Lemma 3.4]{B1999}).
We define a ccp map
\[
\Psi _{m, n, m_0,\delta}\colon M_{\mathcal{L}(m, n, m_0)} \otimes \overline{\left(\mO_A\right)^{\Ga}} \to \left(\mO_A\right)^{''}
\]
by
\[
\Psi _{m, n, m_0,\delta }\left(\sum_{\mu, \nu \in \mathcal{L}(m, n, m_0)}e_{\mu, \nu } \otimes X_{\mu, \nu} \right) \coloneqq  \sum_{\mu, \nu \in \mathcal{L}(m, n,m_0)} \xi(|\mu|)\overline{\xi (|\nu|)}S_\mu X_{\mu, \nu} S^* _\nu, \quad X_{\mu, \nu} \in \overline{\left(\mO_A\right)^{\Ga}}.  
\]
The following is our analog of \cite[Lemma 3.4]{Matsumoto2002}.
\begin{lem}\label{lem:1221}
For any $f\in \Omega_A(m_0,n)$, $l\in \{0,\ldots ,m-1\}$ and $\delta >0$ with $m,n,m_0\in \N$, it holds that
 \[
\left\| \Psi_{m', n, m_0,\delta} \circ \Phi_{m', n, m_0 }\left(\gamma _A^l f \right) -\gamma _A^l (f) \right\| \leq \frac{\delta}{2}.
 \]
 \end{lem}
 \begin{proof}
Let $f\in \Omega_A(m_0,n)$ and $l\in \{0,\ldots ,m-1\}$.
It follows from the observation around \eqref{eq:1221c} that there exists $p\in \N$ with $\vert p\vert \le m_0$ such that $E_q(\gamma _A^l(f)) =\gamma _A^l(f)$ if $q=p$ and $=0$ otherwise.
Hence,
 \begin{align*}
\Psi_{m', n, m_0,\delta} \circ \Phi_{m', n, m_0}\left(\gamma _A^l (f) \right)&=\sum_{\mu, \nu \in \mathcal{L}(m, n,m_0)} \xi(|\mu|)\overline{\xi (|\nu|)}S_\mu S^*_\mu E_{|\mu| -|\nu|}(\gamma _A^l(f))S_\nu S^* _\nu\\
&=
\sum_{\substack{\mu, \nu \in \mathcal{L}(m, n,m_0) \\ \vert \mu \vert - \vert \nu \vert =p}} \xi(|\mu|)\overline{\xi (|\nu|)}S_\mu S^*_\mu \gamma _A^l(f)S_\nu S^* _\nu\\
&=\sum_{\substack{s, r\in\{m,\ldots , m+m_0\} \\  s- r =p}} \xi(s)\overline{\xi (r)} \left(\sum _{\mu  \in \overline{L(s, n)}}S_\mu S^*_\mu \right) \gamma _A^l(f) \left(\sum _{\nu  \in \overline{L(r, n)}}S_\nu S^* _\nu \right).
\end{align*}
Therefore, it follows from Proposition \ref{prop:1117} that
 \begin{align*}
\Psi_{m', n, m_0,\delta} \circ \Phi_{m', n, m_0}\left(f \right)&
=\gamma _A^l(f)  \sum_{\substack{s, r\in\{m,\ldots , m+m_0\} \\  s- r =p}} \xi(s)\overline{\xi (r)}.
\end{align*}
Since $\Vert \gamma _A\Vert =1$ and $\Vert f\Vert =1$, by the construction of $\xi$, this implies the conclusion.
 \end{proof}

Fix an integer $n_1$ and a minimal $A$-orthogonal division $H^{(1)}, \cdots, H^{(m_A)}$ such that $H^{(1)}, \cdots, H^{(P_A)}$ are infinite sets and
\[
\N \backslash \{1, \cdots, n\} \subset \bigcup_{l=1}^{P_A}H^{(l)}
\]
for any $n\ge n_1$. 

\begin{prop}\label{prop:5.6}
Assume that {\tt (SH)}, {\tt (FS)}, {\tt (AF)} and {\tt (PCF)} hold for $A$.
Assume also that there exists a KMS state $\varphi$ for the canonical gauge action $\Ga$.
Then one has 
\[
ht(\ga_A) \leq \sup_{n \ge \max\{n_0,n_1\}}\left(\limsup_{m\to \infty} \frac{1}{m}\log \overline{w(m, n)}\right).
\]
\end{prop}

\begin{proof}
The proof is similar to Proposition \ref{mainprop1}.
We show that for all $\delta >0, m_0 \in \N$ and $n \geq \max\{n_0,n_1\}$ one has
\[
\limsup_{m \to \infty} \frac{1}{m} \log rcp\left(\bigcup_{k=0}^{m-1} \ga_A ^{k}\left(\Omega_A(m_0, n)\right); \delta\right) \leq \limsup_{m\to \infty} \frac{1}{m}\log \overline{w(m, n)}.
\]
Let $\mathcal{F}_n$ be the commutative C$^*$-subalgebra of $\left(\mO_A\right)^{''}$ generated by $\{ S^* _\kappa S_\kappa : \kappa \in  \overline{L(1, n)}\}$.
As in the proof of Proposition \ref{mainprop1}, one can find ccp maps
\[
\phi_2 \colon \left(\mO_A\right)^{''} \to M_{m_2}(\C), \quad \psi_2 \colon M_{m_2}(\C) \to B(\mathcal{H}_\varphi )
\]
with
\[
\| \psi_2\phi_2\left(S_\kappa ^* S_\kappa\right) - S_\kappa ^* S_\kappa \| \leq \delta \left(2 \overline{w(1, n)}\right)^{-1}
\]
for any $\kappa \in \overline{L(1, n)}$.
Write $m'\coloneqq  m+m_0$ and set 
\[
\phi\coloneqq  (id_{M_{{\mathcal{L}(m', n, m_0)}}} \otimes \phi_2)\Phi_{m', n, m_0} \ , \ \psi \coloneqq \Psi _{m', n, m_0}(id_{M_{{\mathcal{L}(m', n, m_0)}}} \otimes \psi_2).
\]
See the following diagram, and compare it with the diagram in Section \ref{ss:1goa}.
\vspace{0.1cm}
$$\xymatrix@M=4pt{
\mbox{\footnotesize $\mathcal O_A$} \ar@{->} [rr] ^(0.45){\Phi_{m',n,m_0}}  
\ar@{->} [ddrrr] _(0.4){\phi} & & \ar @{} [d] |{} 
\mbox{\footnotesize $\Phi_{m',n,m_0}(\mathcal O_A)$}
\ar@{->} [ddr] ^{id \otimes \phi_2} \ar@{} [rr] ^{\empty}
 &  & 
\ar@{} [dr] < 8pt> |{} \mbox{\footnotesize $\empty$}
 \ar@{} [rr] ^{\empty} 
\ar@{} [d] &  & \mbox{\footnotesize $B (\mathcal H_\varphi)$}  \\ 
&   &  &   & \mbox{\footnotesize $M_{\mathcal L(m',n,m_0)}  \otimes \mathcal O_A$}
\ar @{} [d]  |{}  \ar@{->} [urr] _{\Psi_{m,n,m_0,\delta}} &  & \\
         &     &   & 
\mbox{\footnotesize $M_{\mathcal L(m',n,m_0)} \otimes M_{m_2} $}  
\ar@{->} [ur] _(0.55){id \otimes \psi_2} 
\ar@/_4pc/ @{->} [uurrr] _(0.65){\psi} &   &   &  }$$

\vspace{0.1cm}

The similar arguments to the proof of Proposition \ref{mainprop1} imply that for any $f \in \Omega_A(m_0, n), m\in \N$ and $l=0, \cdots, m-1$, one has
\[
\left\| (id_{M_{\mathcal L(m', n, m_0)}} \otimes \psi_2 \phi_2) \left(\Phi_{m', n, m_0}\left(\ga_A ^l(f)\right)\right) - \Phi_{m', n, m_0}\left( \ga_A ^l(f) \right) \right\| \leq \frac{\delta}{2}.
\]
Combining it with Lemma \ref{lem:1221}, we obtain
\begin{multline*}
\left\| \psi \phi \left(\ga_A ^l(f) \right) - \ga_A ^l(f) \right\| \\
 \leq \left\| \Psi _{m', n, m_0,\delta }(id_{M_{{\mathcal{L}(m', n, m_0)}}} \otimes \psi_2 \phi_2) \left(\Phi_{m', n, m_0}\ga_A ^l(f)\right) - \Psi_{m', n, m_0,\delta }\Phi_{m', n, m_0}\ga_A ^l(f) \right\| \\
+ \left\| \Psi_{m', n, m_0,\delta } \circ \Phi_{m', n, m_0}\left(\ga_A ^l(f) \right) - \ga_A ^l(f) \right\| 
 \leq \delta.
\end{multline*}
This inequality implies that we have
\[
\frac{1}{m} \log rcp\left(\bigcup_{k=0}^{m-1} \ga_A ^{k}\left(\omega_A(m_0, n)\right); \delta\right) \leq \frac{1}{m}\log \left(m_2\mathcal{W}(m+m_0, n, m_0)\right)
\]
for any $m \in \N$.
Thus, we get
\[
\limsup_{m \to \infty} \frac{1}{m} \log rcp\left(\bigcup_{k=0}^{m-1} \ga_A ^{k}\left(\omega_A(m_0, n)\right); \delta\right) \leq \limsup_{m\to \infty} \frac{1}{m}\log \mathcal{W}(m, n, m_0).
\]
Moreover, we have
\[
\mathcal{W}(m, n, m_0) = \sum_{k=m}^{m+m_0}\overline{w(k, n)} \leq (m_0+1) \overline{w(m+m_0, n)}.
\]
This implies
\begin{equation*}
\begin{split}
\limsup_{m \to \infty} \frac{1}{m} \log rcp\left(\bigcup_{k=0}^{m-1} \ga_A ^{k}\left(\omega_A(m_0, n)\right); \delta\right) &\leq \limsup_{m\to \infty} \frac{1}{m}\log \mathcal{W}(m, n, m_0) \\
&\leq \limsup_{m\to \infty} \frac{1}{m}\log \overline{w(m, n)}.
\end{split}
\end{equation*}
This completes the proof.
\end{proof}

Now  Theorem \ref{thm:1}  immediately  follows from \eqref{eq:1211b}, \eqref{eq:1211c} and Proposition \ref{prop:5.6}.

\section{Proof of Theorem \ref{thm:2}}\label{s:lower}

We prove Theorem \ref{thm:2} passing through the diagram in Subsection \ref{subsec31}, which is a consequence of Proposition \ref{prop:wconti}.
 Assume that {\tt (SH)}, {\tt (UCF)}, {\tt (FS)}  hold for $A$. 
 We divide our proof into four steps.

{\bf Step 1: Variational Principle for commutative entropies.} 
It follows from \cite{Sarig1999}  for potential 0 (more originally from the Ruessian literature \cite{Gurevich1969}) that if $\sigma _A: \Sigma _A\to \Sigma _A$ is topologically mixing,  then it holds that
\[
h_{\mathrm{G}}( \sigma _A) =\sup \left\{ h_\mu (\sigma _A) \,\big\vert\, \text{$\mu$: $\sigma_A$-invariant Borel probability measure}\right\}
\]
(recall Section \ref{ss:Ge}).
In the following step, we will obtain an inequality that relates each metric entropy of $\mu$ and the corresponding non-commutative metric entropy.

{\bf Step 2: Commutative/non-commutative metric entropies.} 
Given a probability measure $\mu$ on $\Sigma _A$, one can define a state $\varphi _\mu$ on the C$^*$-algebra $C(X_A)$ by
\[
\varphi _\mu(f) =\int _{\Sigma_A} \iota_*fd\mu 
\]
(recall \eqref{eq:1228a} for $\iota$).
If $\mu$ is a $\sigma_A$-invariant probability measure, then
\[
\varphi _\mu \left((\hat{\sigma} _A)_*f \right) = \int _{\Sigma_A} f\circ \hat{\sigma} _A\circ \iota\ d\mu 
=\int _{\Sigma_A} f\circ \iota\circ \sigma _A d\mu  = \int _{\Sigma_A} f\circ \iota\ d\mu
=\varphi_{\mu}(f).
\]
That is, $\varphi _\mu$ is a $(\hat{\sigma} _A)_*$-invariant state.
Recall that the Sauvageot--Thouvenot entropy $h_{\varphi_{\mu}}^{ST}((\hat{\sigma}_A)_*)$ for the C$^*$-dynamical system $(C(X_A),\varphi_{\mu},(\hat{\sigma}_A)_*)$ (see Definition 5.1.1 in \cite{NS2006}) is given by
\[
h_{\varphi_{\mu}}^{ST}\left(\left(\hat{\sigma}_A\right)_*\right)=\sup\left\{h_{\nu}(\xi;T)+\sum_{Z\in\xi}S\left(\lambda\left(\,\cdot\,\otimes 1_Z\right),\varphi_{\mu}\right)\right\}\footnote{Here $h_{\nu}(\xi;T)$ is the entropy of a finite partition $\xi$ with respect to $(X,\nu,T)$ and $S$ is the relative entropy of positive linear functionals on $C(X_A)$, which is non-negative for all states on $C(X_A)$.},
\]
where the supremum is taken over all stationary couplings $\lambda$ (i.e., $\lambda$ is a $((\hat{\sigma}_A)_*\otimes (T)_*)$-invariant state on $C(X_A)\otimes L^{\infty}(X,\nu)$ with $\lambda\vert_{C(X_A)}=\varphi_{\mu}$ and $\lambda\vert_{L^{\infty}(X,\nu)}=\nu$) of $(C_0(X),\varphi_{\mu},(\hat{\sigma}_A)_*)$ with abelian dynamical systems $(X,\nu,T)$ and over all finite measurable partitions $\xi$ of $X$.
We can connect the metric entropy $h_{\mu}(\sigma_A)$ with the corresponding non-commutative metric entropy $h_{\varphi_{\mu}}((\hat{\sigma}_A)_*)$ over $C(X_A)$ via $h_{\varphi_{\mu}}^{ST}((\hat{\sigma}_A)_*)$.
Note that $\lambda=\varphi_{\mu}\otimes \mu$ provides a stationary coupling of $(C(X_A),\varphi_{\mu},(\hat{\sigma}_A)_*)$ with $(\Sigma_A,\mu,\sigma_A)$.
Thus we have
\[
h_{\mu}\left(\sigma_A\right) \le h_{\varphi_{\mu}}^{ST}\left((\hat{\sigma}_A)_*\right).
\]
Then it follows from Theorem 5.1.5 in  \cite{NS2006} that for each $\sigma _A$-invariant probability measure $\mu$ on $\Sigma_A$,
\[
h_{\varphi_{\mu}}^{ST}\left(\left(\hat{\sigma}_A\right)_*\right)
=h_{\varphi _\mu}\left(\left(\hat{\sigma} _A\right)_*\right)
\]
since $C(X_A)$ is nuclear and hence $\varphi_{\mu}$-approximating net exists.
Thus we have
\begin{multline*}
\sup \left\{ h_\mu (\sigma _A) \, \big\vert \, \text{$\mu$:  $\sigma_A$-invariant probability measure}\right\}\\
\le \sup \left\{ h_\varphi \left((\hat{\sigma} _A)_*\right) \, \big\vert \, \text{$\varphi$:  $\left(\hat{\sigma}_A\right)_*$-invariant state}\right\}.
\end{multline*}

{\bf Step 3:
 Variational Principle inequality for non-commutative entropies.}
Note that the statement in Proposition 6.2.7 of \cite{NS2006} holds true, even when we consider ucp self maps instead of automorphisms.
Hence if $\gamma $ is a ucp map on an exact C$^*$-algebra $\mathcal O$, then it holds that
\[
h_\varphi (\gamma ) \leq ht (\gamma )
\]
for any $\gamma$-invariant state $\varphi$. 
Therefore we have
\[
 \sup \left\{ h_\varphi \left(\left(\hat{\sigma} _A\right)_*\right) \, \big\vert \, \text{$\varphi$:  $(\sigma_A)_*$-invariant state}\right\}
\le ht \left(\left(\hat{\sigma} _A\right)_*\right).
\]

{\bf Step 4: Completion of the inequality.} 
As a summary of Step 1--Step 3, we have already proved
\[
h_G\left(\sigma_A\right)
\le ht \left(\left(\hat{\sigma} _A\right)_*\right).
\]
Recall that it follows from Proposition \ref{commprop} that
\[
\left(\hat{\sigma}_A\right)_* \circ \Phi = \Phi \circ \gamma_A.
\]
Since the Gelfand--Naimark representation $\Phi$ is an isomorphism and $\mathcal D_A\subset \mathcal O_A$ is a $\gamma _A$-invariant subalgebra, 
\[
ht \left(\left(\hat{\sigma} _A\right)_*: C(X_A) \to C(X_A)\right) = ht \left(\gamma _A: \mathcal D_A \to \mathcal D_A\right) \leq ht (\gamma _A: \mathcal O_A \to \mathcal O_A)
\]
(recall Section \ref{subsec23}).
Combining all the estimates, we get the desired inequality in Theorem \ref{thm:2}.

\appendix
 
\section{Relation with subshifts  and Matsumoto algebras}
In this appendix, we briefly recall the canonical way  to construct a countable Markov shift from a class of mixing subshifts with a finite alphabet  and discuss about the difference between the non-commutative entropies of the canonical cp map of the Matsumoto algebras associated with subshifts and our canonical cp map of the Exel--Laca algebras.

\subsection{Hofbauer's Markov extension}\label{app:B1}
Let $T: [0,1] \to [0,1]$ be a piecewise monotonic map, that is,  there exist disjoint intervals $I_1, \ldots , I_k$  such that $ [0,1] \setminus \bigcup _{j=1}^kI_j$ is a finite set, denoted by $F$, and  $T \vert _{I_j}$ is monotonic and continuous for each $1\leq j \leq k$.
 Assume that $T$ is topologically mixing.
Consider the maximal continuity set 
 \[
 X_T=[0,1] \setminus \bigcup _{n=0}^\infty T^{-n} (F),
\]
and the coding map $\mathcal I : X_T \to \{ 1, \ldots ,k\} ^{\mathbb N}$ given by
\[
(\mathcal I(x) )_j = \ell \quad \text{if $T^{j-1}(x) \in I_\ell$} .
\]
We note that $\mathcal{I}$ is well-defined and injective since $T$ is mixing (in particular, transitive).
Denote the closure of $\mathcal I(X_T) $ by $\Sigma _T$, which is a subshift with the alphabet $\{1,\ldots,k\}$.
To indicate the dependence of $T$, we denote the left-shift operation of $\Sigma _T$ by $\sigma _T$.

In what follows we will construct Hofbauer's Markov extension, which is a countable Markov shift with the same entropy as $\Sigma _T$.\footnote{This construction was later extended to any subshift by Buzzi \cite{Buzzi1997}, but for notational simplicity, we here only consider Hofbauer's Markov extension.}
Let $D\subset\Sigma_T$ be a closed subset with $D\subset [i]\coloneqq \{ x=\{x_j\} _{j\in\mathbb N}\in \Sigma _T \, :\,  x_1 =i\}$ for some $1\le i\le k$.
We say that a non-empty closed subset $C\subset\Sigma_T$ is a \textit{successor} of $D$ if $C=[j]\cap\sigma(D)$ for some $1\le j\le k$. 
We define a set $\mathcal{D}$ of vertices by induction. First, we set
$\mathcal{D}_0=\{[1],\ldots,[k]\}$. If $\mathcal{D}_n$ is defined for $n\ge 0$, then we define $\mathcal{D}_{n+1}$ by
\[
\mathcal{D}_{n+1}=
\left \{C \subset\Sigma_T : \text{there exists }D\in\mathcal{D}_n\text{ such that }C\text{ is a
successor of }D \right\}.
\]
We note that $\mathcal{D}_n$ is a finite set for each $n$ since the number of successors of any closed subset of $\Sigma_T$ is at most $k$ by the definition. Finally, we set
\[
\mathcal{D} =\bigcup_{n\ge 0}\mathcal{D}_n.
\]
We define a matrix  $A_T=\{A_{T}(D, C)\}_{(D,C)\in\mathcal{D}^2}$ by
\[
A_{T}(D, C)=
\left\{
\begin{array}{ll}
1 & (D\rightarrow C), \\
0 & (\text{otherwise}).
\end{array}
\right.
\]
Then $\Sigma _{A_T}=\{\{D_i\}_{i\in\mathbb{N}}\in \mathcal{D}^{\mathbb{N}} : D_i\rightarrow D_{i+1} \; \text{for all $i\in\mathbb{N}$} \}$ is a Markov shift with a countable alphabet $\mathcal{D}$ and a transition matrix $A_T$.
We define $\Psi\colon \Sigma _{A_T}\to \{1,\ldots,k\}^{\mathbb{N}}$ by 
\[
\Psi(\{D_i\}_{i\in\mathbb{N}}) =\{x_i\}_{i\in\mathbb{N}}
\quad  \text{for $\{D_i\}_{i\in\mathbb{N}}\in\Sigma _{A_T}$},
\]
where $1\le x_i\le k$ is the unique integer such that $D_i\subset [x_i]$ holds for each $i\in\mathbb{N}$.
Then it is clear that $\Psi$ is continuous, countable-to-one, and satisfies $\Psi\circ\sigma=\sigma\circ\Psi$.
We remark that $\Sigma _{A_T}$ is not mixing in general although $\Sigma_T$ is mixing.

\begin{thm}$\mathrm{(}$\cite[Theorem 11]{Hofbauer1986}$\mathrm{)}$
\label{irreducible}
Suppose that $h_{\rm top}(\sigma   _T)>0$. Then we can find a subset $\mathcal{C}\subset\mathcal{D}$ such that $\Sigma _{A_{T,\mathcal{C}}}$ is mixing and $\Psi(\Sigma _{A_{T,\mathcal{C}}})=\Sigma_T$, where $A_{T,\mathcal{C}} =\{ A_{T,\mathcal{C}}(D,C)\}_{(D,C)\in\mathcal{C}^2}$ denotes the submatrix of $A_T$ for $\mathcal{C}\subset\mathcal{D}$.
\end{thm}

Since  topological entropy does not change by a countable-to-one semi-conjugacy, it holds that $h_{\mathrm{top}}(\sigma _T) = h_{\mathrm{top}}(\sigma _{A_{T,\mathcal{C}}}) $.
It is known that $A_{T,\mathcal{C}}$ is finite if and only if $\Sigma _T$ is sofic (cf.~\cite{Buzzi2006}).

An important example of a non-sofic shift $\Sigma _T$ induced by a piecewise monotonic map $T$ is the \emph{$\beta$-shift} (i.e.~$T(x) = \beta x \mod 1$) with a non-algebraic $\beta >1$.
In fact, when $\beta$ is not algebraic, there is an increasing sequence $\{ i_n\}_{n\in \mathbb N} \subset \mathbb N$ (determined by the so-called kneading sequence of $T$) and an integer $k$ such that the transition matrix $A_T=\{ A_T(i,j)\}_{(i,j)\in \mathbb N^2}$ is of the form 
\[
A_T(i,j)=
\begin{cases}
1 \quad &((i,j)=(i_n,1) \; \text{for some $n$} \;\; \text{or} \;\; i=j-1)\\
0 \quad &(\text{otherwise})
\end{cases} 
\]
for each $i, j \ge k$ (cf.~\cite{Buzzi2006}).
Comparing it with \eqref{eq:renewal} (together with Example \ref{example:1}), one may easily see that the transpose matrix $A_T^{tr}$ of $A_T$ satisfies {\tt (SH)}, {\tt (FS)}, {\tt (AF)}, {\tt (SI)}, {\tt (SD)} and {\tt (O)}.

\subsection{Comparison with the canonical cp map of Matsumoto algebras}\label{app:B2}
Given a subshift $\Sigma \subset \{ 1,\ldots ,k\}^{\mathbb N}$, Matsumoto introduced in \cite{Matsumoto2002} a canonical way to construct a $\lambda$-graph system $\mathcal L(\Sigma)=(V, E, \lambda ,\iota)$,\footnote{Here,  a multiple $\mathcal L=(V, E, \lambda ,\iota)$ is called a $\lambda$-graph system if it consists of a vertex set $V= \bigcup _{l\in \mathbb N} V_l$,  an edge set $E=  \bigcup _{l\in \mathbb N} E_{l,l+1}$, a labeling map $\lambda : E\to \{1,\ldots ,k\}$ and surjective maps $\iota =\{\iota _{l,l+1}\}_{l \in \mathbb N}$, $\iota _{l,l+1}: V_{l+1}\to V_l$, such that (1) each $V_l$, $E_{l,l+1}$ is a finite set, (2) every vertex in $V$ has outgoing edges and every vertex in $V$, except in $V_1$, has incoming edges, and (3)  there exists an edge in $E_{l,l+1}$ with label $i$ and terminal $v \in V_{l+1}$ if and only if there exists an edge in $E_{l-1,l}$ with label $i$ and terminal $\iota (v)\in V _l$.}
which satisfies
\[
\Sigma =\{ \{ \lambda (e_j)\}_{j\in \mathbb N} \, :\,e_j\in E_{j-1,j}, \; t(e_j)=s(e_{j+1})\},
\]
where $s(e)$, $t(e)$ are source and terminal vertexes of an edge $e$, respectively.
The \emph{Matsumoto algebra} $\mathcal O_{\mathcal L(\Sigma)}$ is the universal unital C$^*$-algebra generated by partial isometries $\{ S_j\} _{j=1,\ldots ,k}$ and projections $\{E_{m}^l\} _{m=1,\ldots ,M(l), \, l \in \mathbb N}$ satisfying that
\begin{itemize}
\item $\sum _{j=1}^k S_j S_j^* =1$;
\item $\sum _{m=1}^{M(l)} E_m^l =1$,  $E_n^l =\sum _{m=1}^{M(l+1)}I_{l,l+1}(n,m) E_m^{l+1}$;
\item $(S_i S_i^* )E_n^l = E_n^l(S_i S_i^* )$;
\item $S_i^*  E_n^lS_i^*  =\sum _{m=1}^{M(l+1)} A_{l,l+1}(n,i,m) E_m^{l+1}$,
\end{itemize}
for any $1\le i\le k$, $1\le n\le M(l)$ and $l\in \mathbb N$, where
\begin{itemize}
\item[$\empty$] $A_{l,l+1}(n,i,m)=
\begin{cases}
1 \quad &(s(e)=v_n^l, \; \lambda (e) =i ,\; t(e)=v_m^{l+1} \; \text{for some $e\in E_{l,l+1}$})\\
0 \quad &(\text{otherwise})
\end{cases}$;
\item[$\empty$]
$I_{l,l+1}(n,m)=
\begin{cases}
1 \quad &(\iota(v_m^{l+1}) =v_n^l)\\
0 \quad &(\text{otherwise})
\end{cases}.$
\end{itemize}
Furthermore, the canonical cp map $\gamma _{\mathcal L(\Sigma)}: \mathcal O_{\mathcal L(\Sigma)}\to \mathcal O_{\mathcal L(\Sigma)}$ is defined by
\begin{equation}\label{eq:canonicalcpsubshift}
\gamma _{\mathcal L(\Sigma)}(T)\coloneqq  \sum _{j=1}^k S_j T S_j^* \quad (T\in \mathcal O_{\mathcal L}).
\end{equation}
Then, the following estimate of the non-commutative entropy of $\gamma _A$ holds in terms of the volume entropy  $h_{\mathrm{vol}}(\mathcal L(\Sigma)) \coloneqq  \limsup _{l\to\infty} \frac{1}{l} \log \vert P_l(\mathcal L(\Sigma))\vert$ of $\mathcal L(\Sigma)$,
where $P_l(\mathcal L(\Sigma))$ is the set of all labeled paths starting at $V_1$ and terminating at $V_l$.
\begin{thm}$\mathrm{(}$\cite[Corollary 3.6]{M2005}$\mathrm{)}$
It holds that
\[
ht (\gamma _{\mathcal L(\Sigma)} ) = ht (\gamma _{\mathcal L(\Sigma)} \vert _{\mathcal D_{\mathcal L(\Sigma)}}) = h_{\mathrm{vol}}(\mathcal L(\Sigma)), 
\]
where $\mathcal D_{\mathcal L(\Sigma)}$ is the commutative C$^*$-subalgebra of $\mathcal O_{\mathcal L(\Sigma)}$ given by
\[
\mathcal D_{\mathcal L(\Sigma)} =\overline{\mathrm{span}}\left\{ S_\alpha E_m^l S^* _\alpha : \alpha \in \Sigma ^*,  \; m=1,\ldots,M(l), \; l\in \mathbb N \right\}.
\]
\end{thm}
Matsumoto showed that if $\Sigma$ is a sofic shift or a $\beta$-shift,  then $h_{\mathrm{vol}}(\mathcal L(\Sigma)) = h_{\mathrm{top}}(\sigma _\Sigma)$.
On the other hand, if $\Sigma$ is a Dyck shift or Motzkin shift, then $h_{\mathrm{vol}}(\mathcal L(\Sigma)) > h_{\mathrm{top}}(\sigma _\Sigma)$.
He also constructed a continuous left-shift operation $\tilde{\sigma} _{\Sigma }$ on a compact Hausdorff space consisting of infinite words, associated with a given subshift $\Sigma$, such that $h_{\mathrm{vol}}(\mathcal L(\Sigma)) = h_{\mathrm{top}} (\tilde{\sigma} _{\Sigma })$ (\cite[Section 2]{M2005}).
So, it follows from the above theorem that when $\Sigma $ is the Dyck shift,
\[
ht (\gamma _{\mathcal L(\Sigma)} \vert _{\mathcal D_{\mathcal L(\Sigma)}}) =  h_{\mathrm{top}} (\tilde{\sigma} _{\Sigma }) > h_{\mathrm{top}}(\sigma _\Sigma).
\]
On the other hand, due to the results in Sections \ref{subsec23} and \ref{subsec31}, if $A=\{ A(i,j)\}_{(i,j)\in \mathbb N^2}$ satisfies {\tt (SH)},  {\tt (UCP)}, {\tt (FS)} (that the Markov extension of Dyck shifts satisfies; cf.~\cite{TY2022}), then it holds that
\[
ht (\gamma _{A} \vert _{\mathcal D_{A}})  = h_{\mathrm{top}}(\hat \sigma _A) \ge  h_{G}(\sigma _A).
\]

\section*{Acknowledgments}

This work was partially supported by JSPS KAKENHI
Grant Numbers 19K14575 and 21K20330.

\bibliographystyle{siam}
\bibliography{bibliography}

\end{document}